\documentclass[12pt]{article}
\usepackage{amsmath,amsthm,amssymb,latexsym,color}
\usepackage{bbm}
\theoremstyle{plain}
\newtheorem{theor}{Theorem}[section]
\newtheorem{claim}[theor]{Claim}
\newtheorem{prop}[theor]{Proposition}

\newtheorem{cor}[theor]{Corollary}
\newtheorem{lemma}[theor]{Lemma}
\newtheorem{rem}[theor]{Remark}
\theoremstyle{remark}

\def\R{{\mathbb R}}
\def\Prob{{\mathbb P}}

\def\row{R}

\def\supp{{\rm supp }}

\def\Event{{\mathcal E}}

\def\f{{\mathcal F}}

\def\spn{{\rm span}\,}

\def\Net{{\mathcal N}}

\def\d{{\rm dist}}
\def\Om0{\Omega_0}
\def\Inc{S}
\def\Cl{\widetilde{C}_H}
\def\HH{{\mathcal H}}
\def\omep{\Omega_{2m,\eps}}

\newcommand{\p}{\mathbb{P}}
\newcommand{\E}{\mathbb{E}}

\newcommand{\lam}{\lambda}
\newcommand{\eps}{\varepsilon}
\def\la{\left\langle}
\def\ra{\r\rangle}

\def\nx{\| x \|}
\def\r{\right}

\newcommand{\Mc}{\mathcal{M}_{n,d}}

\newcommand{\il}{I^\ell}
\newcommand{\ir}{I^r}

\newcommand{\st}{\mathcal{T}}

\def\O1{\Omega_1(\eps_1)}
\def\idmat{{\rm Id}}

\def\nn{n_3}
\def\ww{z}
\def\WW{W}
\def\hh{h}
\def\BB{\mathcal{B}}

\def\lll{\ell_0}

\def\bb{{b_\st}}
\def\C{\mathbb{C}}
\def\a{a_1}
\def\aa{a_2}
\def\aaa{a_3}

\title{The smallest singular value of a shifted $d$-regular random square matrix}

\author{
Alexander E. Litvak
\and
Anna Lytova
\and
Konstantin Tikhomirov
\and
Nicole Tomczak-Jaegermann
\and
Pierre Youssef
}
\newcommand\address{\noindent\leavevmode

\medskip
\noindent
Alexander E. Litvak
and Nicole Tomczak-Jaegermann,\\
Dept.~of Math.~and Stat.~Sciences,\\
University of Alberta, \\
Edmonton, AB, Canada, T6G 2G1.\\
\texttt{\small
e-mails:  aelitvak@gmail.com \, \, and \, \,
nicole.tomczak@ualberta.ca}\\

\medskip

\noindent
Anna Lytova,\\
Faculty of Math., Physics, and Comp. Science,\\
University of Opole,\\
48, Oleska str., 45-052,\\
Opole, Poland.\\
\texttt{\small
e-mail:
alytova@uni.opole.pl
}\\

\medskip

\noindent
Konstantin Tikhomirov,\\
Dept.~of Math.,
Princeton University,\\
Fine Hall, Washington road,\\
Princeton, NJ 08544.\\
\texttt{\small
e-mail:   kt12@princeton.edu}\\

\medskip

\noindent
Pierre Youssef,\\
Universit\'e Paris Diderot,\\
Laboratoire de Probabilit\'es, Statistique et Mod\'elisation,\\
75013 Paris, France.\\
\texttt{\small
e-mail:  youssef@math.univ-paris-diderot.fr}
}

\date{}

\begin{document}

\maketitle

\begin{abstract}
We derive a lower bound on the smallest singular value of a random $d$-regular matrix,
that is, the adjacency matrix of a random $d$-regular directed graph.
Specifically, let $C_1<d< c n/\log^2 n$ and let $\Mc$ be the set of all $n\times n$ square matrices with $0/1$ entries, such
that each row and each column of every matrix
in $\Mc$ has exactly $d$ ones. Let $M$ be a random matrix uniformly distributed on $\Mc$.
Then the smallest singular value $s_{n} (M)$ of $M$
is greater than $n^{-6}$ with probability at least $1-C_2\log^2 d/\sqrt{d}$,
where $c$, $C_1$, and $C_2$ are absolute positive constants
independent of any other parameters.
Analogous estimates are obtained for matrices of the form $M-z\,\idmat$, where $\idmat$ is the identity matrix and
$z$ is a fixed complex number.
\end{abstract}

\noindent
{\small \bf AMS 2010 Classification:}
{\small
primary: 60B20, 15B52, 46B06, 05C80;
secondary: 46B09, 60C05
}


\tableofcontents

\section{Introduction}
\label{s:intro}

The present paper belongs to a
sub-area of the random matrix theory often called
{\it  non-limiting} or {\it non-asymptotic} (see e.g. \cite{LedT,
RV-ICM}). Development of this direction of research was
motivated by some problems in statistics, compressed sensing and
computer science in general, as well as in asymptotic geometric analysis.
The object of the study is a large random matrix of a fixed size, and
a typical goal is to obtain quantitative probabilistic estimates for
its eigenvalues or singular values in terms of dimension of the
matrix.
In this paper we avoid a discussion
of corresponding {\it limiting} results, and refer, in particular, to
books \cite{BS2010,EY} and references therein for more information
(see also \cite{DavS} for interplay between limiting and non-limiting
results and for applications).

The study of the non-limiting behaviour of the smallest and the
largest singular values is a very important research direction.
Recall that for an $m\times n$ ($m\geq n$) matrix $A$, the largest and
the smallest singular values can be defined as
\[
  s_{1}(A) = \|A\| =\max_{\|z\|_2=1} \|Az\|_2 \quad \mbox{ and } \quad
  s_n(A)= \min_{\|z\|_2=1} \|Az\|_2,
\]
where $\|A\|$ denotes the operator norm of $A$ acting from $\ell_2^n$
to $\ell_2^m$ (also called the spectral norm).  In case when $m=n$ and
the matrix $A$ is invertible, we have $s_n(A)=1/\|A^{-1}\|$.
The knowledge of the magnitude of the extreme singular values gained
significance in connection with asymptotic geometric analysis,
numerical analysis (in particular, smoothed analysis of the condition number), the
problem of approximating covariance matrices of multidimensional distributions,
the study of delocalization properties of eigenvectors.
 Moreover, for square non-Hermitian matrices, estimating the extreme singular values forms
a crucial step in computing the limit of the empirical spectral
distribution.  We provide
a brief overview of those directions.

First, assume that $A$ is a tall rectangular matrix with independent
rows (satisfying certain conditions). Estimating $s_1(A)$ can be quite
difficult (excluding the subgaussian case, see, for example,
\cite[Fact~2.4]{LPRT}).  The lower bounds for $s_n(A)$ often require
covering arguments, estimates for small ball probabilities,
anti-concentration results, and on many occasions bounds on $s_1(A)$.
For bounds on $s_1(A)$ and $s_n(A)$, we refer to \cite{ALPT, MenP,
  GLPT, T} and references therein.  We would like to notice that
strong estimates for $s_n(A)$ for this model can be obtained bypassing
analysis of $s_1(A)$, and under very weak conditions on the
distributions of the rows \cite{KM,Ol,Y1,Y2} (see also \cite{GeMu} for
related yet different setting).

Another model of randomness, which is closer to the main topic of our
paper, involves square random matrices or matrices with the aspect
ratio $m/n$ very close to one, with i.i.d.\ entries.  In this setting,
obtaining optimal quantitative lower bounds for $s_n(A)$ requires more
delicate arguments, compared to the model considered above.  We refer,
in particular, to \cite{LPRT, TV-ann, RV, RV-comm, SV13, RT} and
references therein (see also \cite{AGLPT} for square matrices with
independent log-concave columns).  In the context of numerical linear
algebra, this research direction is  related to estimating the
condition number of a square matrix.  Recall that the condition number
of an $n\times n$ matrix $A$ is defined as
\[
  \sigma (A)= s_1(A)/s_n(A) = \|A\| \, \, \|A^{-1}\|.
\]
The condition number serves as a measure of precision of certain
matrix algorithms \cite[Chapter~III]{NLA1}, \cite{NLA2}.  The study of
the condition number in the random setting goes back to von~Neumann
and his collaborators (see \cite[pp. 14, 477, 555]{VN1} and
\cite[Section~7.8]{VN2}), whose numerical experiments suggested that
for a random $n\times n$ matrix $A$ one should have
$\sigma (A)\approx n$ with high probability.  In a more general
context, when the spectral norm $\|\cdot\|$ is replaced with an
operator norm $\|\cdot\|_{X\to Y}$ for two $n$-dimensional Banach
spaces $X$ and $Y$, the quantity $\|A\|_{X\to Y}\|A^{-1}\|_{Y\to X}$
plays a crucial role in the local theory of Banach spaces and
asymptotic geometric analysis through its relation to the
Banach--Mazur distance \cite{AGA2, AGA3}.  Estimating the condition
number of a shifted matrix $A+B$ (with $A$ random and $B$ fixed) was
put forward as an important problem by Spielman and Teng \cite{ST02},
in context of smooth analysis of algorithms (see, in particular,
\cite{SST06,TV07b,TV10c}).  As a very important application, the
quantitative lower bounds for $s_n(A+B)$, with $B$ being a complex
multiple of the identity, have been used to establish the circular law
for the empirical spectral distribution in the i.i.d.\ model (see
\cite{TV,BC} and references therein for the historical account of the
problem). Indeed, it is known that using the Hermitization technique,
one needs to show the uniform integrability of the logarithmic
potential with respect to the empirical singular value distribution of
the shifted matrix. Bounding the smallest singular value away from
zero is therefore essential for such method to work. As the limiting
distribution is not the aim of this paper and since the uniform
integrability requires also a control of the remaining singular
values, we leave this for a future investigation (see \cite{LLTTY-str, LLTTY-c}).


The model studied in this paper differs from the ones discussed above
in two crucial aspects.  Let us set up the framework.  Let $d \le n$
be (large) integers, which we assume to be fixed throughout the paper.
Consider the set $\Mc$ of square $n\times n$ matrices with $0/1$ entries
such that each row and each column of a matrix $M\in \Mc$ contains
exactly $d$ ones. Such matrices will be called $d$-regular. These are
adjacency matrices of $d$-regular digraphs (directed graphs), where we
allow loops but do not allow multiple edges. On $\Mc$ we take the
uniform probability measure, turning $\Mc$ into a probability space,
and consider the random matrix distributed according to this measure.
The two main differences from the models mentioned in the
previous paragraphs are complex dependencies between the matrix
entries and (for $d\ll n$) sparsity of the matrix, i.e., large number
of zero entries.

The question of estimating $s_n(M)$ (or, more generally, $s_n(M+B)$
for a fixed matrix $B$), where $M$ is uniformly distributed in $\Mc$,
can be justified in two respects. First, this is a natural model with
complex dependencies between the matrix entries, which does not allow
the use of standard conditioning arguments (such as fixing the span of
$n-1$ rows of a random matrix and studying the conditional
distribution of the distance of the remaining row to the
span). Techniques developed for treating this model can potentially be
adapted to more general models with dependencies. Second, as we show
in this paper, unlike the Erd\H{o}s--R\'enyi random model (see below for
the definition and a more detailed comparison), the $d$-regularity
condition
guarantees strong lower bounds on $s_n(M)$ with large probability even
in the case when $d\ll\log n$ when the corresponding Erd\H{o}s--R\'enyi
adjacency random matrix with the parameter $p=d/n$ is singular with
large probability. This provides a better understanding as to what
causes singularity of sparse random matrices (``local'' obstructions
to invertibility such as a zero row in the Erd\H{o}s--R\'enyi model
versus ``global'' obstructions when the non-trivial null vectors have
many non-zero components).


Singularity of adjacency matrices of uniform random $d$-regular
digraphs was first considered by Cook in \cite{Cook-adjacency}.  He
adapted to the case of directed graphs a conjecture of Costello and Vu
from \cite[Section~10]{CV}, which asserted that for $3\leq d \leq n-3$
with probability going to 1 as $n$ goes to infinity the adjacency
matrix of a random $d$-regular undirected graph is non singular (see
also Vu's survey  \cite[Problem~8.4]{Vu-surv} and  2014 ICM talks by Frieze
\cite[Problem~7]{Frieze ICM} and Vu \cite[Conjecture~5.8]{Vu}).
 The argument in \cite{Cook-adjacency} was
based on discrepancy properties of random digraphs studied in
\cite{Cook-digraphs}, together with some anti-concentration arguments
and a sophisticated use of the simple switching operation.  It
established non-singularity of the adjacency matrix with a large
probability for $d\geq C \log ^2 n$.

The question about singularity of adjacency matrices in the case
$d\leq \log n$ remained open, moreover it was not clear whether the
condition $d\gg \log n$ comes from limitations of the method used in
\cite{Cook-adjacency} or if a random matrix uniformly distributed on
$\Mc$ becomes singular in this regime.  As we mentioned above, in the
Erd{\H{o}}s--R\'enyi model, a random matrix is singular with
probability close to one in the case $d\ll\log n$.  In \cite{LLTTY:15}
(see also \cite{LLTTY-cras}), the authors of the present paper were
able to partially answer this question by showing that a random
$d$-regular matrix is non-singular for all $d$ bigger than a large
universal constant, however, the probability of the singularity was estimated
from above by a negative power of $d$ (see also \cite{LLTTY-rank}, where we
proved  that the rank of such matrices should be at least $n-1$ with
probability going to one as $n$ grows to infinity).  The main novelty of
\cite{LLTTY:15} compared to \cite{Cook-adjacency} rested on three new
ingredients -- a particular version of the covering argument which is
applied to study the structure of the kernel of random matrices, on a
different set of properties of random digraphs, and on a new
approach to anti-concentration results.

However, both papers \cite{Cook-adjacency} and \cite{LLTTY:15} didn't
provide any quantitative estimates.
Combining methods from \cite{Cook-adjacency} and \cite{LLTTY:15} with
an elaborate chaining argument, in recent papers \cite{Cook-circ} and
\cite{BCZ}, quantitative lower bounds on the smallest singular value
of the adjacency matrix were proved for the uniform and permutation
models, under an assumption that $d$ is polylogarithmic in $n$.
Moreover, considering shifted adjacency matrices, the authors of
\cite{Cook-circ,BCZ} were able to obtain the circular law for the
eigenvalue distribution (again, for $d$ at least polylogarithmic in
$n$).  Precisely, in $\cite{Cook-circ,BCZ}$ it was shown that, with
some conditions on the shift $\WW,$ the smallest singular value
$s_{n}(M+\WW)$ of a random shifted matrix is at least $n^{-C\log_d n}$
with probability close to one.
 Still papers
$\cite{Cook-circ,BCZ}$ do not provide any bounds for $s_{n}$
when $d$ is growing slower than $\log n$ and moreover, even for $d$
growing faster than $\log n$ but subpolynomial in  $n$, they don't
provide a polynomial in $n$ bound for $s_n$.

\smallskip

The goal of the present paper is to provide polynomial in $n$ lower
bounds on the smallest singular value of a random matrix uniformly
distributed on $\Mc$ for
$d$ larger than a (fixed large) absolute constant.
Our approach results in better bounds not only for small $d$ but for the
entire range $C\leq d\leq cn/\log^2 n$.  Our main result is the
following theorem, in which we also allow shifts of random matrices
for the sake of future applications (see also Remark~\ref{sharpb} for
more precise bounds).

\begin{theor}\label{mainth}
  There are universal constants $C, c>0$ with the following property.
  Let $C<d< c n/((\log n)(\log \log n))$. Then for every $\ww\in \C$
  with $|\ww|\leq d/6$ one has
  \[
    \Prob\Big\{M\in\Mc:\, s_{n}(M-\ww\idmat)\geq n^{-6}
    \Big\}\geq 1- C\log^2 d/\sqrt{d}.\]
\end{theor}

It is natural to compare our model with the Erd{\H{o}}s--R\'enyi
model, i.e.~matrices whose elements are i.i.d.~Bernoulli $0/1$
variables with the expectation $d/n$.  Intuitively one would expect
that $d$-regular matrices should behave in a similar way to the
Erd{\H{o}}s--R\'enyi model.  This in turn seems to be similar (after
applying a proper normalization $\sqrt{d/n}$) to random $\pm 1$
matrices, where values $1$ and $-1$ appears with probability
$1/2$. Since for the latter model one has $s_{n} \approx 1/\sqrt{n}$,
we would expect the answer $s_{n}\approx \sqrt{d}/n$ for both
$d$-regular matrices and for the Erd{\H{o}}s--R\'enyi model. Indeed,
the Erd{\H{o}}s--R\'enyi model was recently treated in \cite{BasRud},
where it was proved that with high probability
$s_{n}\geq c (d/n)^{-c/\log d} \sqrt{d}/n$, provided that
$c\log n \leq d \leq n-c\log n$. Note that if $d$ is polynomial in $n$
then this gives the expected bound $\sqrt{d}/n$. However, there is one
delicate point in such a comparison. It is easy to see that for
$d< \log n$ a matrix in the Erd{\H{o}}s--R\'enyi model has a zero row
with probability more than half, therefore more than half of
matrices in this model are singular. To the contrary, our theorem
shows that in the case of $d$-regular matrices most matrices are
non-singular. In particular, this means that the regularity prevents a
matrix from being singular, in a sense reducing the randomness.

\smallskip

The remaining part of the introduction is devoted to a brief
description of main ideas and to a short overview of the proof of
Theorem~\ref{mainth}.
An often employed approach to estimating the smallest
singular value (in other words, to bounding $\|Mx\|_2/\|x\|_2$
from below for every non-zero
$x\in \C^n$) is to partition $\C^n$ and work separately with different types of
vectors. The idea to split the Euclidean sphere into two parts goes
back to Kashin's work \cite{kashin} on an orthogonal decomposition of
$\ell^{2n}_1$, where the splitting was defined using the ratio of
$\ell_2$- and $\ell_1$-norms. A similar idea was used by Schechtman
\cite{Schech} in the same context. In the context of the smallest
singular value one usually splits $\C^n$ into vectors of smaller
complexity (close to sparse vectors) and ``spread'' vectors (in
particular, with a relatively small $\ell_\infty$-norm). Such a
splitting was introduced in \cite{LPRT} (see also \cite{LPRTV}) and
was further formalized later in \cite{RV} into a concept of
``compressible'' and ``incompressible'' vectors in $\C^n$.
Compressible vectors are essentially
vectors of smaller dimension, so the set of compressible vectors has a relatively small
complexity.
Therefore, using the standard $\eps$-nets argument
and the union bound one can obtain good bounds for $\|Mx\|_2/\|x\|_2$ for all compressible vectors.
For incompressible vectors, the question can be reduced to estimating the distance between a column
of the matrix and the linear span of remaining columns, which is in turn bounded using Littlewood--Offord--type
inequalities.

In our model, due to special structure of the matrices (in particular,
due to the lack of independence and due to the sparsity of a matrix) the concept of
compressible and incompressible vectors is not directly applicable. In \cite{Cook-adjacency}, Cook
replaced these notions with another type of structural dichotomy, namely sparse vectors were replaced with a
bigger class of vectors having at least one large level set (where ``large" means {\it of cardinality at least $n/d^c$})
while {\it unstructured vectors} were the ones with small level sets. These notions we also used in
\cite{BCZ, Cook-circ, LLTTY:15}. In \cite{LLTTY:15} the {\it structured vectors} were referred to as {\it almost
constant vectors}, since there ``a large level set" meant {\it of cardinality at least} $n-n/\log d$.
Thus,  an almost  constant vector is a very sparse vector shifted by a constant vector.
In the present  work, we further refine this splitting in order to take full advantage
of the discrepancy properties of random $d$-regular matrices.
Specifically, we define four (overlapping) classes on $\C^n$, which we call {\it steep
vectors}, {\it gradual vectors} (that is non-steep), the {\it almost
constant vectors}, and the {\it essentially non-constant
vectors} (that is the complement of almost
constant vectors).  Roughly speaking, almost constant vectors are those with
many coordinates almost equal to each other. The gradual vectors are
vectors $x=(x_i)_i\in \C^n$, whose sequence $(x_i^*)_i$ (a
non-increasing rearrangement of $(|x_i|)_i$) has a regular decay,
i.e., has no significant jumps, where by {\it a jump} we mean
$x_k^*\gg x_m^*$ for some $k\ll m$.
The steep vectors are vectors
possessing such jumps.


The idea to consider steep vectors comes from the following observation.
A steep vector $x$ possesses a ``steep'' jump in its non-increasing rearrangement which induces a partitioning of
the graph into the set of vertices (indices) corresponding to large coordinates and those associated with small coordinates.
Now the expansion properties of the random $d$-regular graph automatically imply that the set of vertices corresponding to large coordinates has
a large neighborhood (many vertices are connected to the set)
and, moreover, many vertices are simply (i.e., through one edge) connected to the set. In terms
of the $d$-regular matrix, this translates into having relatively many rows which have exactly one entry equal to one within the set of columns
corresponding to large coordinates of $x$.
Then the inner product of each such row with the steep vector $x$ is large by absolute value thanks
to the big ratio of the magnitudes of large and small coordinates of $x$: the inner product is dominated by the value of the unique large coordinate of $x$ corresponding to the aforementioned non-zero component of the selected row.
The implementation of this naive idea is more involved and requires a careful selection of the size of the jump
(responsible for the magnitude of the inner product) and its location (which must take into consideration the
graph expansion properties).
To this, another difficulty adds up, lying in the construction of the associated nets as one needs
to balance the size of the net with the individual probability bounds.
The actual argument is technically involved
since we are required to distinguish several types of jumps as well as different jump locations and
combine these with very delicate construction of the $\eps$-nets.
It will be further discussed in
Section~\ref{steep}.

Bounding the magnitude of the matrix-vector product for almost constant gradual vectors is
straightforward. First, notice that the $\ell_2$-norm of an almost constant gradual vector is comparable to
the $\ell_2$-norm of its ``constant part''. Moreover, employing properties of random $d$-graphs, one can show that
there are many rows for which most of their support lies on the ``almost
constant'' part of the vector. This further implies that the inner product of such rows with the vector is separated from zero,
and knowing the $\ell_2$ norm of the vector thus provides uniform quantitative lower bounds on the
product of our random matrix with almost constant gradual vectors.

After we obtain bounds for the above two classes it remains to deal with
essentially non-constant vectors.
Using general algebraic properties of square matrices
we reduce the problem of estimating the smallest singular number to
estimating distances between rows (or columns) of the matrix and
certain subspaces (similar reductions were used in \cite{RV, TV}). More precisely,
we consider quantities of the form
$$
\d\big(\row_i(M), \spn\big\{\{\row_k(M)\}_{k\neq i,j},\, \row_i(M)+\row_j(M)\big\}\big),
$$
for pairs of indices $1\leq i\neq j\leq n$. The first observation is that the subspace appearing above is invariant under simple switchings
between the $i$-th and $j$-th rows.
Then  we condition on a realization
of the subspace $\{\row_k(M)\}_{k\neq i,j}$ and consider all matrices sharing
the same realization of this subspace (which will form an equivalence class).
We use the randomness of the remaining two rows to control the inner product of a normal vector to the subspace
with the $i$-th row. A key observation is that for most pairs of rows,
the restriction of an essentially non-constant vector to the support of those
two rows remains ``non-constant''. This step requires two properties of random $d$-regular digraphs which we
proved in \cite{LLTTY:15}. We show that within the equivalence class the inner product of
the $i$-th row with an essentially non constant vector can be viewed as a sum of independent
random variables, to which anti-concentration inequalities can be
applied. This strategy was developed in \cite{Cook-adjacency, Cook-circ, LLTTY:15}
and is further refined in this work by splitting each class above into subclasses on which the same normal vector can be used
to have a control on the smallest singular value.

\bigskip

\noindent
{\bf Acknowledgement.}
We are grateful to an anonymous referee for careful reading the first draft of the
manuscript and many valuable suggestions, which helped us to improve presentation.
The second and the third named authors would like to thank
University of Alberta for excellent working conditions
in January--August 2016,
when a significant  part of this work was done.

\section{Preliminaries}
\label{preliminaries}

By ``universal'' or ``absolute'' constants we always mean numbers
independent of all involved parameters,
in particular independent of $d$ and $n$. When we say that a parameter (or a constant) is sufficiently large
(resp. sufficiently small) it means that there exists an absolute positive constant such that the corresponding statement
or inequality holds whenever the parameter is larger (resp. smaller) than this absolute constant.
Given positive integers $\ell<k$, we denote the sets
$\{1, 2, \ldots , \ell\}$ and $\{\ell, \ell + 1,  \ldots , k\}$ by $[\ell]$ and $[\ell, k]$, respectively.
For any two real-valued functions $f$ and $g$ we write $f\approx g$ if there are two absolute positive constants
$c$ and $C$ such that $cf\le g\le Cf$. By $\idmat$ we denote the $n\times  n$ identity  matrix.
For $I\subset [n]$, let $I^c:=[n]\setminus I$ denote the complement of $I$ in $[n]$
and let $P_I$ denote the operator of orthogonal projection on the coordinate subspace $\C^I$.
For a vector $x=(x_1,\dots, x_n)\in \C^n$, we denote its
$\ell_{\infty}$-norm by $\|x\|_{\infty}=\max_i |x_i|$ and its $\ell_2$-norm by $\|x\|_2$.
We denote also $\bar x =(\bar x_i)_{i= 1}^n$, where $\bar z$ is the complex conjugate of $z\in \C$,
and by $(x_i^*)_{i}^n$ we denote the non-increasing rearrangement of the sequence $(|x_i|)_{i=1}^n$.
We use $\la\cdot, \cdot \ra$ for the standard inner product on $\C^n$, that is
$\la x, y \ra = \sum _{i= 1}^{n} x_i \bar y_i$.
The  unit ball of the complex space $\ell _\infty^n$ is denoted by $B_\infty^n$.
To simplify notation, we don't distinguish between row and
column vectors, this will always be clear from the context. In particular, for an $n\times n$ matrix $U$ and
a vector $x\in \C^n$, we have
$$
 U x=(\langle R_i(U) , \bar x  \rangle)_{i=1}^n \quad \mbox{ and } \quad x \, U =\sum _{i=1}^n x_i R_i(U),
$$
where $R_i(U)$, $i\leq n$, denote rows of $U$.
By  $\|U\|$ we denote the operator norm of $U$, considered as a linear operator $U$ from (complex) $\ell_2$ to $\ell_2$.
Note also that by the Perron--Frobenius theorem for
every $M\in \Mc$ one has $\|M\|=d$.

\smallskip

We will use the following  anti-concentration Littlewood--Offord type lemma
(\cite{Er}, see also \cite{Kl}).
\begin{prop}
\label{p:Erdos}
Let $\xi_1$, $\xi_2$, ..., $\xi _m$ be independent $\pm 1$ Bernoulli random variables and let $x_1$, $x_2$, ..., $x _m$
be complex numbers such that $|x_i|\ge 1$, $i\le m$. Then for every $t\geq 1$ one has
$$
 \sup_{a\in \C}\p\Big(\big|\sum_{i=1}^m\xi_i  x_i-a\big|< t\Big)\leq \frac{C_{\ref{p:Erdos}}\,t}{\sqrt{m}},
$$
where $C_{\ref{p:Erdos}}>0$ is a universal constant.
\end{prop}

The next lemma is a ``quantified'' version of Claim~4.7 from \cite{LLTTY:15}.

\begin{lemma}\label{l: quantified claim4.7}
Let $x=(x_1,x_2,\dots,x_m)\in\C^m$ be a vector such that for some $\rho>0$ and $\eps \in(0,1)$ we have
$$\forall\lambda\in\C\quad \big|\big\{i\leq m:\,|x_i-\lambda|\geq \rho\big\}\big|\geq \eps m.$$
Then there are disjoint subsets $J$ and $Q$ of $[m]$ such that
$$|J|,|Q|\geq  \eps m/4 \quad\mbox{ and }\quad \forall i\in J,\,\forall j\in Q\;\;|x_i-x_j|\geq \rho/\sqrt{2}.$$
\end{lemma}
\begin{proof}
Let $y^1:=\mathrm{Re}(x)$ and $y^2:=\mathrm{Im}(x)$ be the real and imaginary part of $x$,
respectively. First, observe that there is $k\in\{1,2\}$ such that
\begin{equation}\label{eq: aux 087ga}
\forall\lambda\in\R\quad \big|\big\{i\leq m:\,|y^k_i-\lambda|\geq \rho/\sqrt{2}\big\}\big|\geq \eps m/2.
\end{equation}
Indeed, assume the opposite, i.e., that there exist real numbers $\lambda_1$ and $\lambda_2$ such that
$$\big|\big\{i\leq m:\,|y^k_i-\lambda_k|\geq \rho/\sqrt{2}\big\}\big|< \eps m/2,\quad k=1,2.$$
Then for $\lambda:=\lambda_1+i\,\lambda_2$ we necessarily have
$$\big|\big\{i\leq m:\,|x_i-\lambda|\geq \rho\big\}\big|< \eps m,$$
contradicting the assumption of the lemma.

Without loss of generality, we can assume that condition \eqref{eq: aux 087ga} holds for $k=1$,
and that the coordinates of $y^1$ are arranged in the non-increasing order.
Denote $p:=\lceil \eps m/4\rceil$.
Set $J:=\{1,2,\dots,p\}$ and $Q:=\{m-p+1,\dots,m\}$.
Clearly, it is enough to show that $y^1_{p}\geq \rho/\sqrt{2}+y^1_{m-p+1}$.
Assume the opposite. Then the set $I:=\{p,\dots,m-p+1\}$
has cardinality strictly greater than $m-\eps m/2$, and for $\lambda:=y^1_{p}$ we have
$|y^1_i-\lambda|< \rho/\sqrt{2}$ for all $i\in I$ contradicting \eqref{eq: aux 087ga}. The result follows.
\end{proof}

\smallskip

\smallskip

We will need the following simple combinatorial claim about relations.
Let $A$, $B$ be sets, and $R\subset A\times B$ be a relation.
Given $a\in A$ and $b\in B$, the image of $a$ and preimage of $b$ are defined by
$$
     R(a) = \{ y \in B \, : \, (a,y)  \in  R\} \quad \mbox{ and } \quad
     R^{-1}(b) = \{ x \in A \, : \, (x,b)  \in  R\}.
$$
We also set $R(A)=\cup _{a\in A} R(a)$.
We have the following standard estimate (see e.g. Claim~2.1 in \cite{LLTTY:15}).

\begin{claim} \label{f}
Let $s, t >0$.
Let $R$ be a relation between two finite sets $A$ and $B$ such that for
every $a\in A$ and every $b\in B$ one has $|R(a)|\geq s$ and $|R^{-1}(b)|\leq t$.
Then
$
        s |A|\leq  t |B|.
$
\end{claim}

\smallskip

We turn now to properties of $d$-regular matrices.
Recall that $\Mc$ denotes the set of all $n\times n$  $0/1$-valued matrices having
sums of elements in every row and in every column equal to $d$
(the set corresponds to adjacency matrices of directed $d$-regular graphs where we allow loops but do not
allow multiple edges).
Given $n\times n$ matrix $U=(v_{ij})$ we denote its
$i$'th row by $\row_i(U)$ and
$\supp\, R_i(U)=\{j\leq n\,:\, u_{ij}\ne 0\}$. For $J\subset [n]$ we also denote
\begin{equation}\label{SJ}
  S_J:=\{i\le n\,:\, \supp\,R_i(M)\cap J\neq\emptyset \},
\end{equation}
that is, $S_J$ is the union of supports of columns indexed by $J$.

Given $k\leq n$ and $\eps \in (0,1)$, let
\begin{equation*}
\Omega_{k, \eps}:=\Big\{M\in \Mc \,:\,\forall J\subset[n] \, \, \text{ with }\, \,
|J|=k
 \, \, \text{ one has }\, \,|S_J|\geq
(1-\varepsilon) d k \Big\}.
\label{Oo}
\end{equation*}
Clearly, if $k=1$ then $\Omega_{k, \eps}=\Mc$.
The following theorem
is essentially  Theorem~2.2 of \cite{LLTTY:15} (see also Theorem 3.1 there).

\bigskip

\begin{theor}
\label{graph th known}
Let $e^8<d\leq n$, $\varepsilon_0=  \sqrt{\log d/d}$, and  $\varepsilon\in [\varepsilon_0,1)$.
Let $k\leq c_{\ref{graph th known}}\eps n/d$,
where $c_{\ref{graph th known}}\in (0, 1)$ is a sufficiently small absolute positive constant.
Then
\begin{equation*}
\mathbb{P}(\Omega_{k,\varepsilon})\geq 1-\exp\left(-\frac{\eps^2 dk}{8}
\log\left(\frac{e\eps c_{\ref{graph th known}} n}{k d}\right)\right),
\end{equation*}
in particular,
\begin{equation*}
\mathbb{P}\left(\bigcup _{k\leq c_{\ref{graph th known}}\eps n/d}\Omega_{k,\varepsilon}\right)
\geq 1-\left( C_{\ref{graph th known}} d/\eps n\right)^{\eps^2 d/8 }.
\end{equation*}
\end{theor}

We will need two more results from \cite{LLTTY:15}. The following is \cite[Proposition~3.3]{LLTTY:15}.

\begin{prop}[Row and columns are almost disjoint]\label{p:disjoint-rows}
Let $\varepsilon\in (0,1)$ and $8\leq d\leq \eps n/6$.
Denote
\begin{align*}
\Omega_1(\eps):=\Big\{M\in \Mc:\,\forall i,j \in [n]\,
 \quad &|\supp(R_i(M)+R_j(M))|\geq 2(1-\varepsilon)d
 \\
 &\mbox{and}\quad\big|\supp\big(R_i\left(M^T\right)+R_j\left(M^T\right)\big)\big|\geq 2(1-\varepsilon)d \Big \}.
\end{align*}
 Then
 \begin{equation*}
 \mathbb{P}(\Omega_1(\eps))\geq 1- {n^2}\, \left(\frac{e d}{\varepsilon n}\right)^{\varepsilon d}.
\end{equation*}
\end{prop}

\smallskip

Given $0\leq \alpha,\beta\leq 1$, denote by  $\Om0(\alpha,\beta)$ the set of matrices in $\Mc$ having a
zero submatrix of size at least $\alpha n\times \beta n$, that is
\begin{align*}
\Om0(\alpha,\beta):=\{M\in\Mc\,\, :\, \, &\exists I,J\subset[n]\quad\text{such that}\quad
                     |I|\geq \alpha n,\,|J|\geq \beta n, \notag
\\
&\text{ and }\quad \forall i\in I\, \forall j\in J\quad \mu_{ij}=0\}.
\end{align*}

The next result is   Theorem~3.4  from \cite{LLTTY:15} (note that
the condition $\beta \leq 1/4$ there  can be removed by adjusting absolute constants).

\begin{prop}[No large zero submatrices]
\label{p:zero minor}
There exist absolute positive constants $c,  C$
such that the following holds. Let $2\leq d\leq  n/24$,  $0<\beta\leq 1$, and
 $0<\alpha\leq \min(\beta, 1/4)$.  Assume that
$$
\alpha \geq  \frac{C  \log (e/\beta)}{d}.
$$
Then
$$
\mathbb{P}\left(\Om0(\alpha,\beta)\right) \leq
\exp\left(-c \alpha \beta dn\right).
$$
\end{prop}

\bigskip

We now discuss another property of matrices in $\Omega_{m,\eps}$.
We start with the following construction.
Given two disjoint sets $J^\ell$, $J^r\subset[n]$
 and a matrix $M\in \Mc$, denote
$$
\il=\il(M, J^\ell, J^r):=\{i\le n :\,|\supp R_i\cap J^\ell|=1 \, \, \text{ and }\, \,\supp R_i\cap J^r=\emptyset\}
$$
and
$$
 \ir=\ir(M, J^\ell, J^r):=\{i\le n :\,\supp R_i\cap J^\ell=\emptyset\,\,\text{ and }\,\,|\supp R_i\cap J^r|=1\}.
$$
The sets $J^\ell$, $J^r$ will always be clear from the context.
The upper indexes $\ell$ and $r$ refer to {\it left} and {\it right}, since later, given
a vector $x\in \R^n$ with $x_1\geq x_2\geq \ldots \geq x_n\geq 0$, we choose $J^\ell=[k_1]$ and
$J^r=[k_2, n]$ for some $k_1<k_2$ (this is the reason why the above formulas for $\il(M), \ir(M)$ are
asymmetric).

\begin{lemma}\label{c:SJ}
Let $p\ge 2$, $m\geq 1$ be  integers satisfying $pm \leq c_{\ref{graph th known}} \eps n/d$
and let  $J^\ell, J^r\subset[n]$ be such that
$J^\ell\cap J^r=\emptyset$, $|J^\ell|=m$, $|J^r|=(p-1)m$. Let $M\in \Omega_{pm, \eps}$.
Then
$$
|S_{J^\ell}\setminus S_{J^r}|\ge(1-\eps p)d|J^\ell|  \quad \mbox{ and }  \quad
|\il|\ge(1-2\eps p)d|J^\ell|,
$$
where $S_{J^\ell},S_{J^r}$ are defined by \eqref{SJ}.
In particular, if $|J^r|=|J^\ell|=m$  then
$$
   (1-4\eps) dm \leq \min(|\il|, \, |\ir|) \leq \max (|\il|, \, |\ir|) \leq dm .
$$
\end{lemma}

\begin{proof}
Since $M\in\Omega_{pm, \eps}$, we observe that $|S_{J^\ell}\cup S_{J^r}|\ge(1-\eps)pd|{J^\ell}|$. Hence,
$$
 |S_{J^\ell}\setminus S_{J^r}|=|S_{J^\ell}\cup S_{J^r}|-|S_{J^r}|\ge(1-\eps)pd|{J^\ell}|-(p-1)d|{J^\ell}|
 =(1-\eps p)d|{J^\ell}|,
$$
which proves the first estimate. To prove the second one, set
$$
  k:=|\{i\in S_{J^\ell}\setminus S_{J^r}\,:\,|\supp R_i\cap {J^\ell}|=1\}|.
$$
Then the number of ones in the submatrix
$$
 \{\mu_{ij}\,:\, i\in S_{J^\ell}\setminus S_{J^r},\, j\in {J^\ell}\}
$$
is at least
$$
   k + 2(|S_{J^\ell}\setminus S_{J^r}|-k)\ge 2(1-\eps p)d|{J^\ell}|- k .
$$
On the other hand, it cannot exceed $|{J^\ell}|d$. Therefore
$$
  k \ge 2(1-\eps p)d|{J^\ell}|-d|{J^\ell}|=(1-2\eps p)d|{J^\ell}|.
$$
This completes the first part of the lemma. The second one follows by applying
these estimates with $p=2$, using that the roles of $\il$ and $\ir$ are interchangable
and that each row contains exactly $d$ ones.
\end{proof}

We would like to mention that the use of events like $\Omega_{k, \eps}$,
$\Omega_1(\eps)$,
$\Om0(\alpha,\beta)$
for the invertibility problems
for $d$-regular matrices goes back to \cite{Cook-digraphs, Cook-adjacency}, which contained weaker versions of our
Theorem~\ref{graph th known} and Propositions~\ref{p:disjoint-rows} and \ref{p:zero minor}.


\section{Almost constant vectors}
\label{steep}

In this section we treat almost constant vectors, which we split into almost constant gradual
vectors (i.e., vectors with many coordinates almost equal to each other and without jumps)
and almost constant steep vectors (i.e., almost constant vectors with jumps).
First, in Theorem~\ref{lem: a-c}, we prove a bound for almost constant gradual
vectors. This case is less involved and was discussed in the introduction. Then we turn to steep vectors.
Recall that steep vectors possess a significant jump, where by {\it a jump} we mean $x_k^*\gg x_m^*$ for some $k\ll m$.
We split a vector in pieces and check if a jump occurs inside those pieces. We distinguish three types of steep vectors,
$\st_0$, $\st_1$, $\st_2$, according to the place where the first jump occurs -- we introduce parameters
$1\leq n_1< n_2<n_3<n$ and $\st_0$ (which is empty if $n_1=1$) corresponds to the case $1\leq k < m \leq n_1$,
$\st_1$ corresponds to $n_1\leq k < m \leq n_2$, and $\st_2$ corresponds to the case $n_2\leq k < m \leq n_3$
(see precise definitions below).

 When the first jump occurs at the beginning of the sequence $(x_i^*)$, that is for
 vectors in $\st_0$, we force the bound by a large jump only,
so the proof in this case is more deterministic and does not require
an approximation -- for every ``good'' matrix we have a good uniform
bound on vectors having a large jump. More precisely, for such vectors
we use properties of $d$-regular graphs and their adjacency matrices,
which we obtained in \cite{LLTTY:15}. Using these properties, we prove
that with high probability a random $d$-regular matrix has many rows
with only one $1$ in columns corresponding to the first $k$
coordinates and no other ones till the $m$-th column (see Lemma~\ref{c:SJ}). Thus, the inner
product of such a row with $x$ can be bounded as difference of the
absolute value of one ``large'' coordinate and the sum of absolute
values of $d-1$ ``small'' coordinates. Therefore, if we have a jump of
order, say, $4d$, this inner product is separated from zero.
This works when $m/k\lesssim 1/\eps _0 = \sqrt{d/\log d}$. The use of Lemma~\ref{c:SJ} leads to the
restriction $n_1\leq \eps_0 n/d\approx n/d^{3/2}$ (in fact, to have better bounds, we choose $n_1$
even smaller -- of order $n/d^2$). This scheme works for all vectors in $\st_0$ -- we don't
need to assume that vectors are almost constant.

If the first jump occurs later, i.e. for vectors in $\st _1$ and $\st_2$,
the main idea is to use the
union bound, that is, to estimate the probability for an individual
vector with a jump, to construct a good $\eps$-net for such vectors,
and to approximate each such vector by a vector from the
$\eps$-net. The fact that the operator norm of our matrices is $d$  and
our choice of $n_1$ lead to the choice of $\eps =1/d^{3/2}$ for $\eps$-nets
(we need to have a  negative power of $d$).
In this scheme the most important is to have the ``right''
balance between the size of the net and the individual probability
bound. For individual probability bounds we use anti-concentration
type technique together with switching argument, standard in dealing
with $d$-regular graphs. Jumps are needed to apply anti-concentration
and to show that, for a fixed vector $x$ and a fixed index $i$,
matrices having small inner product of $i$-th row with $x$ belong to a
certain class, to which we can apply the switching argument. For this
argument a constant jump, that is $x_k^* > 4 x_m^*$, would be
enough.  Note that the smaller the jump and the larger the ratio $m/k$
the better for us, since we need to have a control of the ratio
$x_1^*/x_m^*$, which is responsible for both, for the final bound on
the singular value and for the size of the net. Note also that
contrary to results for matrices with i.i.d. entries we have to employ
anti-concentration inequalities already for these vectors of
relatively small complexity. Nets will be constructed in $\ell_\infty$
metric fixing $x_k^*=1$ (with $k=n_1$ or $n_2$) in order to control values
of each coordinate indexed between $k$ and $m$.  To have a reasonable size of the net, we
also work with pieces of a vector and approximate each piece
separately. This delicate construction allows us to significantly
decrease the size of the net (in comparison with the standard
constructions). Unfortunately, the size of the net is still quite
large and requires additional restrictions. First, it works only when
$m/k\lesssim d/\log d$, that is we must have both $n_3/n_2$ and $n_2/n_1$
to be at most $d/\log d$. Moreover, since in the individual bounds
Lemma~\ref{c:SJ} is again involved (with a different choice of parameters),
we have an additional restriction $n_2 \lesssim n/d$. This explains our choice
of $n_2\approx n/d$ and hence $n_1\approx n (\log d)/d^2$ and
$n_3\approx n/\log d$. Second, in the case
$n_1\lesssim k<m \lesssim n_2$, to kill a large part of
coordinates (in order to decrease the size of the net) we need a jump
of order $1/\eps = d^{3/2}$. This will lead to the definition of $\st_1$.
Note that again our proof works for vectors from $\st_1$ without an 
additional assumption that vectors are almost constant.

For the part of coordinates with $k \approx n/d$
and $m \approx n/\log d$, corresponding to the definition of $\st_2$,
due to the method used in the proof of Theorem~\ref{lem: a-c}, we cannot
 use a large jump and has to deal with a constant jump. With such a small jump,
without additional restrictions,
the size of the net would be too large to be ``killed" by individual probabilities
bounds.  To overcome this issue, we intersect steep vectors from $\st_3$ with almost constant vectors. This
significantely reduces the ``dimension'' of vectors
(making them essentially one-dimensional on the set of coordinates
corresponding to the ``almost constant part") and allows good
bounds on the size of the net even with a constant jump.

\subsection{Almost  constant, steep, and gradual  vectors: definitions and main results}

\bigskip

To define almost constant and steep vectors we will use the following parameters.
In order to use Theorem~\ref{graph th known}, we fix $\eps_0$ and a related parameter
$p$ as follows:
$$
\eps _0 = \sqrt{(\log d)/{d}}, \quad \quad  p= \lfloor 1/(5 \eps_0) \rfloor=\left\lfloor \tfrac{1}{5} \sqrt{d/\log d}\r\rfloor
$$
(the choice of $p$ comes from $\eps _0 p < 1$ needed in Lemma~\ref{l:T0} in order to apply  Lemma~\ref{c:SJ}).
We also fix three absolute positive sufficiently small constants $\a$, $\aa$, and $\aaa$, satisfying
\begin{equation} \label{new-par-aa}
\aaa \leq \aa /28\leq \a/28^2,
\end{equation}
(we don't try to estimate the actual values of $a_i$'s, the conditions on how small they are will be appearing
in the corresponding proofs).
Set
$$
  n_0:=\lceil \a n \log d/d^2 \rceil, \quad n_2 :=\lfloor \aa n/d \rfloor, \quad \mbox{ and } \quad \nn :=\lfloor \aaa n/\log d \rfloor.
$$
If $n_0=1$, set $n_1=1$. Otherwise,
fix an integer $r\geq 0$ such that $p^{r}<n_0\leq p^{r+1}$ and set
$$
 n_1=\left\{
  \begin{array}{cc}
   n_0, & \mbox{ if }\, n_0\leq p,\\
   p^{r+1}, & \text{otherwise}.
  \end{array}
     \right.
$$
Note that
\begin{equation}\label{new-par-nn}
  \frac{n_2}{n_1}\leq \frac{\aa d}{\a \log d}, \quad \quad
  \frac{n}{n_2}\leq \frac{2 d}{\aa},\quad \quad
  \frac{n}{n_3}\leq \frac{2 \log d}{\aaa}, \quad \quad
  \frac{n_3}{n_1}\leq \frac{\aaa d^2}{\a \log^2 d}, \quad \quad
\end{equation}
and, in the case  $n_0>1$,
\begin{equation}\label{new-par-nnnn}
 n_1\leq pn_0\leq \a \eps_0 n /5d.
\end{equation}

\smallskip

We are ready now to describe our classes.
First, given $\rho>0$,  we introduce a class of almost constant vectors by
$$
  \BB(\rho) :=\{ x\in \C^n\, : \, \exists \lambda \in \C \,  \mbox{ such that }\,
   |\{ i\leq n \, : \, |x_i - \lambda|\leq \rho\, \|x\|_2\}| >n- \nn\}.
$$

\medskip

The definition of the class of steep vectors is
 more involved and
consists of few steps at which we define sets $\st_0$, $\st_1$, and  $\st_2$. We start with $\st_0$.
If $n_0=n_1=1$  we set $\st _0 =\emptyset$.
 If $n_0>1$, we denote
$$
  \st_{0, 0}:= \{x\in \C^n \,:\,  x_{1}^{*}> 4 d x_{m}^{*}\},
$$
where $m=\min(n_0, p)$. In the case $1<n_0=n_1\leq p$ we set $\st_0= \st_{0, 0}$.
Otherwise,
if  $n_0>p$  we set for $1\leq i\leq r$
$$
  \st_{0, i}:=\{x\in \C^n\,:\,  x\not\in \bigcup_{j=0}^{i-1} \st_{0, j}  \,\, \mbox{ and } \, \,  x_{p^{i}}^{*}> 4 d x_{p^{i+1}}^{*}\}
  \quad \mbox{ and } \quad
  \st_0 =\bigcup _{i=0}^{r}\st _{0, i}.
$$
Finally, we define two more sets of  steep vectors, as
$$
  \st _1 : =\{x\in \C^n\,:\,  x\not\in  \st_{0}
  \,\, \mbox{ and } \, \,  x_{n_{1}}^{*}> d^{3/2}\, x_{n_2}^{*}\}
$$
and
$$
  \st _2 : =\{x\in \C^n\,:\,  x\not\in  \st_{0}\cup \st_1
  \,\, \mbox{ and } \, \,  x_{n_{2}}^{*}> 4\, x_{\nn}^{*}\} .
$$
The vectors from $\st_1\cup \st_2\cup \st_3$ we call steep and all other vectors we call gradual.

\smallskip

We introduce the following functions $h_i$, $0\leq i\leq r+1$,
$$
   \hh_{r+1}
   := \left\{
	\begin{array}{ll}
		  \sqrt{3 p}\,  n_1^{ 2+\alpha_d}
       & \mbox{ if } n_0> p, \\
		 2 d^{3/2}/\sqrt{\log d}  & \mbox{ if }  1<n_0\leq p,
\\
         \sqrt{n}  & \mbox{ if }  n_0=1,
	\end{array}
 \right.
\quad
 \hh_i :=
 \left\{
	\begin{array}{ll}
		 \sqrt{n}        & \mbox{ if } i=0, \\
		 \sqrt{n} + \sqrt{2p}\,  p^{ i(2+ \alpha_d)}  & \mbox{ if }  1\leq i\leq r,
	\end{array}
 \right.
$$
where $\alpha_d=\log 4d/\log p -2$ (note $2\log \log d/\log d\leq \alpha_d \leq 4\log \log d/\log d$ for
large $d$). We also denote
$$
 \bb:=
 \left\{
	\begin{array}{ll}
		 4d^{3/2} \hh_{r+1}       & \mbox{ if } \, \, n_0>1, \\
		 d\sqrt{n}  & \mbox{ if }\, \, n_0=1.
	\end{array}
 \right.
$$

\medskip

In this section we prove  two following theorems. The first one treats almost constant gradual vectors,
the second one treats almost constant sleep vectors (in fact, a  slightly larger class).

\begin{theor}\label{lem: a-c}
Let $d\leq n$ be large enough integers.
Let $0<\rho \leq 1 /(5\, \bb)$ and let $x\in \BB(\rho)\setminus (\st _{0}\cup \st_1\cup \st_2 )$.
Then
for every $M\in \Mc$ and for every $\ww\in \C$ with $|\ww |\leq d/6$ one has
$$
  \| (M - \ww \idmat) x\|_2 \geq \frac{d \sqrt{3n}}{5\, \bb}\,  \, \|x\|_2.
$$
\end{theor}

\begin{theor} \label{t:steep}
There are  absolute constants $C>1>c, c_1>0$ such that the following holds.
Let $C<d<c_1n$ and $0<\rho \leq 1/(d^{3/2}\, \bb)$. Let $\ww\in \C$ be such that $|\ww|\leq d$. Denote
$$
   \st = \st _0 \cup \st_1 \cup (\st_2 \cap \BB(\rho))
$$
and
$$
  \Event_{steep}:=\Big\{M\in\Mc\,:\,\exists\;x\in \st\, \, \, \mbox{ such that }
  \, \, \,
  \|(M-\ww \idmat) x\|_2 < \frac{ \sqrt{ n_2 d} }{25 \bb }\, \nx_2
  \Big\}.
$$
 Then
\begin{equation*}
\label{Psteep}
\Prob(\Event_{steep})\leq    n^{-c \min (\log n, \sqrt{d\log d})}.
\end{equation*}
\end{theor}

\begin{rem}\label{cortwoth} In Section~\ref{s: quantitative} we will use these two theorems in the following way.
Let $\rho = 1/(d^{3/2}\, \bb)$, $|\ww |\leq d/6$,
$$
  \BB _0 (\rho) : = \BB (\rho) \cap \left\{x\in\C^n:\,\|x\|_2=1\r\},
$$
and
$$
  \Event :=\Big\{M\in\Mc\,:\,\exists\;x\in \BB _0 (\rho)\, \, \, \mbox{ such that }
  \, \, \,
  \|(M-\ww \idmat) x\|_2 < \rho^2/16
  \Big\}.
$$
Then Theorems~\ref{lem: a-c} and \ref{t:steep} imply that
$$
\Prob(\Event)\leq    n^{-c \min (\log n, \sqrt{d\log d})} \leq 1/2n^2.
$$
\end{rem}

\begin{rem}\label{unifbou}
Note that
$$
  \frac{d \sqrt{3n}}{5\, \bb}\geq  \frac{ \sqrt{ n_2 d} }{25 d^{3/2} \hh_{r+1}}.
$$
In the proof of Theorem~\ref{t:steep} we show also that
$$
  \frac{ \sqrt{ n_2 d} }{25 d^{3/2} \hh_{r+1}}\geq \hh(d, n),
$$
where
$$
 \hh(d, n) =
 \left\{
	\begin{array}{ll}
		 c d^{-3/2}      & \mbox{ if } n_1=1 \, \, \mbox{(that is, if }\, \a n \leq \frac{d^2}{\log d} \mbox{)}, \\
		  c\sqrt{n}\, d^{-3}(\log d)^{-1/2}  & \mbox{ if } 1< n_1\leq p \, \, \mbox{(that is, if }\, \frac{d^2}{\log d}<
\a n \leq \frac{d^{5/2}}{5\log^{3/2} d}\mbox{)},\\
          c d^{5/4} (\log d)^2 n^{-3/2-\alpha_d}  & \mbox{ if } n_1> p \, \, \mbox{(that is, if }\,
    \a n > \frac{d^{5/2}}{5\log^{3/2} d}\mbox{)}.
	\end{array}
 \right.
$$
\end{rem}

\smallskip

In the proof of both theorems we will use the
comparison of $\ell_2$-norm of a given vector with a fixed coordinate. The next lemma provides
such a bound in terms of the functions $h_i$. Moreover, we also estimate the $\ell_\infty$-norm.
Note that we clearly have $\|x\|_2\le \sqrt{n} \, x_{1}^{*}$ for every $x\in \C^n$.

\begin{lemma} \label{l:norma}
Let $d\leq n$ be large enough and $x\in \C^n$, $x\ne 0$.
If $x\in \st _{0,i}$ for some $0\leq i\leq r$, then
$$
    \|x\|_2 \le  \hh_i\, x_{p^{i}}^{*}.
$$
Moreover,
$$
 \|x\|_2 \le \left\{
\begin{array}{ll}
		  \hh_{r+1}\, x_{n_1}^{*}
       & \mbox{ if } x\notin \st _{0}, \\
		 (\bb/4)\, x_{n_2}^{*} & \mbox{ if }  x\notin \st _{0}\cup \st_1,
\\
         \bb\, x_{n_3}^{*} & \mbox{ if }  x\notin \st _{0}\cup \st_1\cup \st_2.
	\end{array}
\r.
$$
\end{lemma}

\smallskip

\begin{proof}
The case $x\in \st _{0, 0}$ is trivial.

If $1<n_0=n_1\leq p$ then
 $\st_0 = \st_{0,0}$ and thus for $x\not \in \st_0$ we observe
$$
  \|x\|_2^2 = \sum _{i=1}^{n_1-1} (x_{i}^*)^2 + \sum _{i=n_1}^{n} (x_{i}^*)^2 \leq
  16 d^2 n_1 (x_{n_1}^*)^2 + n  (x_{n_1}^*)^2 \leq (16 d^2 p + n)(x_{n_1}^*)^2.
$$
The result follows since  $n_0\leq p$ implies $\a n\leq d^2 p/\log d$ and because $d$ is large enough.

\smallskip

We now assume that $n_0>p$. Let $x\in \st _{0,i}$ for some
$1\leq i\leq r$
or let $x\not \in \st _0$ in which case we set $i=r+1$.
Then for every $j<i$, one has $x\not\in \st_{0, j}$, hence,
assuming without loss of generality  that $x_{p^i}^*=1$, we get
$$
  x_1^*\leq (4d) x^*_{p}\leq (4d)^2 x^*_{p^2}\leq \ldots \leq (4d)^i x_{p^i}^* = (4d)^i = p^{i \log 4d/\log p}.
$$
This implies
\begin{align*}
\|x\|_2^2 &= ((x_{1}^*)^2+\ldots + (x_{p}^*)^2) + ((x_{p+1}^*)^2 + \dots + (x_{p^2}^*)^2)
+ \ldots
\\&\leq
p (4d)^{2 i} + p^2 (4d)^{2(i-1)} + \ldots + p^{i} (4d)^{2} + n
\\&=
\frac{p(4d)^2 ((4d)^{2i} - p^{i})}{(4d)^2-p} + n
  \leq 2 p  (4d)^{2i} +n  = 2p \, p^{2i \log 4d/\log p} + n ,
\end{align*}
which implies the result for $i\leq r$. In the case $i=r+1$, that is, if $x\not\in \st_0$, this gives
$\|x\|_2^2\leq 2p\, n_1^{4+2\alpha_d} + n$. Note that we are in the case $n_0>p$, hence $n_1\geq p^2$.
Using the definition of $n_0$, we observe that $\a n\geq d^2 p/\log d$ and therefore
\begin{equation*}
 n_1^4\geq p^6 n_1  \geq \frac{d^3}{(6 \log d)^3}\, \frac{\a n \log d}{d^2} \geq \frac{\a n\, d}{\log d}
\end{equation*}
which implies for sufficiently large $d$ that $\|x\|_2\leq \sqrt{3p }\, n_1^{2+\alpha_d}$.

If $x\notin \st _{0}\cup \st_1$ then clearly $x_{n_1}^*\leq  d^{3/2} x_{n_2}$, and, if
additionally $n_0=1$, then
$$
  \|x\|_2^2 = \sum _{i=1}^{n_2-1} (x_{i}^*)^2 + \sum _{i=n_2}^{n} (x_{i}^*)^2 \leq
   d^3 n_2 (x_{n_2}^*)^2 +  n  (x_{n_2}^*)^2 \leq  (a_2 d^2 n  +  n ) (x_{n_2}^*)^2 \leq  d^2 n  (x_{n_2}^*)^2/16,
$$
provided $a_2<1/20$ and $d$ is large enough. The case $x\notin \st _{0}\cup \st_1\cup \st_2$ follows as well,
since in this case $x_{n_3}^*\leq 4  x_{n_2}^*$.
This completes the proof.
\end{proof}

\subsection{Proof of Theorem~\ref{lem: a-c}}

We will use the following simple claim.

\begin{claim}\label{almconsvect}
Let $J\subset [n]$, $k=|J|$, and $A>1$. Let $M\in \Mc$. Then
$$
  | \{ i\leq n \, : \, |\supp R_i (M)\cap J| \geq   A k d/n \}| \leq n/A.
$$
\end{claim}

\begin{proof}
 The number of ones in the submatrices indexed by  $[n]\times J$ is $kd$. Thus
$$
   | \{ i\leq n \, : \, |\supp R_i (M)\cap J| \geq   A k d/n \}| \, \cdot  A k d/n \leq kd,
$$
which implies the result.
\end{proof}

\medskip

\begin{proof}[Proof of Theorem~\ref{lem: a-c}.]
Clearly, we may assume $x\ne 0$.
Fix a permutation $\sigma=\sigma_x$ of $[n]$ such that $x_i^*=|x_{\sigma(i)}|$ for $i\le n$.
Note that since $x\not\in \st _{0}\cup \st_1\cup\st_2 \cup  \{0\}$ we have
$x_{\nn}^*\ne 0$.

Fix $\lam _0 =\lam _0(x) \in \C$ such that the cardinality of
$$
   J_1:= \{ i\leq n \, : \, |x_i - \lambda_0|\leq \rho\, \nx_2 \}
$$
is at least  $n- \nn+1$. Therefore there exist $k, \ell$ such that $k\leq \nn <\ell$ and
$\sigma(k), \sigma(\ell) \in J_1$.
By  Lemma~\ref{l:norma},
$$
 \|x\|_2
  \leq \bb\, x_{\nn}^*=\bb \, |x_{\sigma(\nn)}|,
$$
hence
$$
  |\lambda_0|- x_{\nn}^*/5\leq |\lambda_0| - \rho \, \nx_2 \leq  |x_{\sigma(\ell)}|= x_{\ell}^*
  \leq x_{\nn}^*  \leq x_{k}^*  =  |x_{\sigma(k)}|
  \leq |\lambda_0| +  \rho\, \nx_2 \leq |\lambda_0| + x_{\nn}^*/5,
$$
where we also used that $\rho\leq 1/(5\bb)$.
This implies
$$
  (5/6) |\lambda_0|  \leq x_{\nn}^*\leq (5/4) \, |\lambda_0|
$$
(in particular, $|\lam _0|\ne 0$) and,
using again that $\rho\leq 1/(5\bb)$,
$$
 \rho \,   \|x\|_2\leq x_{\nn}^*/5\leq |\lam _0| /4.
$$
 Set
$$
  J_2=\sigma([n_{2}])\setminus J_1 , \quad  J_3 = \sigma([\nn])\setminus (J_1\cup J_2), \quad \mbox{ and }
  \quad J_4=[n]\setminus  (J_1\cup \sigma([\nn])).
$$
Then $|J_3|, |J_4|\leq \nn$, $[n]=J_1 \cup J_2\cup J_3\cup J_4$, and
\begin{equation}\label{condonJ}
\forall j\in J_4 \, \, \, \, |x_j|\leq x_{\nn}^*\leq 5|\lam_0|/4 \quad\quad \mbox{ and } \quad\quad
\forall j\in J_3 \, \, \, \, |x_j|\leq x_{n_2}^*\leq  4 x_{\nn}^* \leq 5|\lam_0|.
\end{equation}

\smallskip

 Now, given a matrix $M\in\Mc$, consider
$$
 I_2 = \{ i\leq n \, : \, \supp R_i (M)\cap J_2  \ne \emptyset \} \, \mbox{ and }\,
 I_\ell = \{ i\leq n \, : \, |\supp R_i (M)\cap J_\ell| \geq   16 \nn d/n \},
$$
for $\ell = 3,4$.
Since $M\in \Mc$ and by Claim~\ref{almconsvect},
we have for small enough $a_2$,
$$
 |I_2|\leq d\, n_{2}\leq n/16
 \quad \mbox{ and } \quad
  |I_\ell|\leq  n/16 \, \, \mbox{ for } \, \, \ell =3,4.
$$
Set $I:=[n]\setminus(I_2\cup I_3\cup I_4\cup \sigma([\nn]))$. Then
$$
  |I| \geq  n - 3 n /16 -  \nn \geq  3 n / 4 \quad \mbox{ and } \quad \forall i\in I \, \, \, \,
  |x_i|\leq x^*_{\nn} \leq (5/4) |\lam _0|.
$$
Moreover, for every $i\in I$, denote $J_\ell '={J' _\ell} (i) = J_\ell \cap \supp R_i (M)$
for $1\leq \ell\leq  4$, and  note that ${J' _2}=\emptyset$ since $i\not\in I_2$.
Using the triangle inequality, we observe for every $i\in I$,
$$
 |\la R_i(M-\ww \idmat), \bar x  \ra|
 \geq \Big|\sum _{j \in  {J_1'}}  x_j\Big| - \sum _{j \in {J_3'}} |x_j| -
 \sum _{j \in {J_4'}} |x_j| - |\ww x_i|.
$$
We estimate  terms in the right hand side separately. By the definition of $J_1$, we have
$$
\Big|\sum _{j \in  {J_1'}}  x_j\Big|\geq |\lam _0|\, |{J' _1}|- \sum _{j \in  {J' _1}}
\Big| x_j-\lambda_0\Big|\geq |{J' _1}| \, ( |\lam _0|- \rho \Vert x\Vert_2)
\geq (d - 32 \nn d/n)\, ( |\lam _0|- \rho \Vert x\Vert_2),
$$
where for the last inequality we used that ${J' _2}=\emptyset$ and that
for $i\not\in I_3\cup I_4$ one has
$$
 |{J' _1}|= d-   |{J' _2}|-|{J' _3}|- |{J' _4}|\geq d - 32 \nn d/n.
$$
Using (\ref{condonJ}), we obtain
$$
  \sum _{j \in  {J_3'}}  |x_j|+ \sum _{j \in  {J_4'}} | x_j|\leq |{J' _3}|\, x_{n_2}^* + |{J' _4}|\, x_{\nn}^*\leq 100 |\lam _0| \nn d/n.
$$
Putting together the above estimates,
 we obtain for large enough $d$
\begin{align*}
    |\la R_i(M-\ww \idmat), \bar x  \ra|
  &\geq (d - 32 \nn d/n)(|\lam _0|- \rho \|x\|_2) - 100 |\lam _0| \nn d/n   - (5/4) |\lam _0| | \ww|
  \\
  &\geq |\lam _0|d/2,
\end{align*}
where we used  $|\lam _0|- \rho \|x\|_2\geq (3/4)|\lam _0|$, $\nn/n\leq c/\log d$, and $|\ww|\leq d /6$.
This implies
$$
   \| (M - \ww \idmat) x\|_2 \geq  \frac{|\lam _0| d}{2}\, \sqrt{\frac{3n}{4}} \geq
   \frac{ d \sqrt{3n}}{5 }\,  x^*_{\nn}\geq \frac{ d \sqrt{3n}}{5 \bb}\, \|x\|_2 ,
$$
which completes the proof.
\end{proof}

\subsection{Lower bounds on $\|Mx\|_2$ for  vectors from $\st_0$}

Here we provide lower  bounds on the ratio $\|Mx\|_2/\nx _2$  for vectors $x$ from $\st_0$.
Recall that given $\eps$ and $k$ the set $\Omega_{k, \eps}$  was introduced  before
Theorem \ref{graph th known}.

\begin{lemma} \label{l:T0}
Let $C\leq d\leq n$, where $C$ is an absolute positive constant and  $x\in \st_0$.
Let $\ww\in \C$ be such that $|\ww|\leq d$.
If $1<n_0=n_1\leq p$ and   $M\in  \Omega_{n_0,\eps _0}$ then
$$ \|(M-\ww Id) x\|_2 \geq \sqrt{d/8n}\, \|x\|_2.$$
If
$$
 n_0>p \quad \quad  \mbox{ and }\quad  \quad M\in\bigcap _{j= 1}^{r+1}\, \Omega_{p^j,\eps _0}
$$
then
$$
 \|(M-\ww \idmat) x\|_2 \geq \min\left\{\sqrt{d/8n},\, \,
 \frac{p\sqrt{n_1 d}}{8 \hh_{r+1}}
 \r\}\, \|x\|_2.
 $$
\end{lemma}

\begin{proof}
We prove the case $n_0\geq p$, the other case is similar. Fix $x\in \st_0$ and
fix $0\leq i\leq r$ such that $x\in \st_{0, i}$ and denote $m=p^i$.
Fix a permutation $\sigma=\sigma_x$ of $[n]$ such that $x_j^*=|x_{\sigma(j)}|$ for $i\le n$.
Then $x_{m}^{*}> 4 d x_{p m}^{*}$. Let
$$
  J^\ell=\sigma([m]), \quad  J^r=\sigma([p m]\setminus[m]), \quad \mbox{ and } \quad
  J_3:=(J^\ell\cup J^r)^c.
$$
Then, for sufficiently small $\a$,
$$
  |J^\ell\cup J^r|= pm \leq p n_0 \le  c_{\ref{graph th known}}\eps_0 n/d \quad \, \, \, \mbox{ and } \quad \, \, \,
 |J^r|=(p-1)|J^\ell|=(p-1)m.
$$
Denote by $I_0$ the set of rows  having exactly one 1 in $J^\ell$ and no 1's in $J^r$.
 Lemma~\ref{c:SJ}  implies that
$$
  |I_0|\geq  (1-2p\eps_0)m d\ge 3m d/5.
$$
Let $I=I_0\setminus (J^\ell\cup J^r)$ (so that the submatrix indexed by  $I\times (J^\ell\cup J^r)$ does not intersect the main diagonal).
Then $|I|\geq 3md/5-pm\geq md/2$ provided that
$d$ is large enough.  By definition, for every $s\in I$  there exists $j(s)\in J^\ell$ such that
$$
   \supp R_{s}\cap J^\ell=\{j(s)\},\quad \supp R_{s}\cap J^r= \emptyset,\quad\text{and}\quad
   \max_{i\in J_3}|x_i|\le x^*_{m p} .
$$
Using Lemma~\ref{l:norma}, the fact that $s\not\in J^\ell\cup J^r$ (which implies $x^*_{s}\leq x^*_{pm}$), and that $j(s)\in J^{\ell}$
(which implies $|x_{j(s)}|\geq x^*_m> 4 d x^*_{m p}$), we obtain
\begin{align*}
  |\langle R_{s} (M-\ww \idmat),\, \bar x \rangle| &=\Big| x_{j(s)}+\sum_{j\in J_3\cap \supp R_{s}} x_j - \ww   x_s \Big| \\&  \geq|x_{j(s)}|- (d-1)\, x_{m p}^*
   - |\ww| \, x_{m p}^* \, \ge  x^*_{m}/2 \geq   \|x\|_2/ 2 \hh_i .
\end{align*}
Since the number of such rows is $|I|\geq m d/2= p^i d/2$ we obtain
$$
  \|(M-\ww \idmat) x\|_2 \geq \sqrt{p^i d}\, \|x\|_2/(2 \sqrt{2} \hh_i).
$$
If $i=0$ then $p^{i/2}/ \hh_i=1/\sqrt{n}$. If $i\geq 1$ and $\sqrt{n}\geq \sqrt{2p}\,  p^{ i\,(2+ \alpha_d)} $,
then $\hh_i\leq 2\sqrt{n}$ and $p^{i/2}/ \hh_i\geq  p^{i/2}/(2\sqrt{n})\geq 1/\sqrt{n}$ provided $d$ is large enough. If $\sqrt{n}\leq \sqrt{2p}\,  p^{ i\,(2+ \alpha_d)} $ then $\hh_i\leq 2\sqrt{2p}\,  p^{ i\, (2+ \alpha_d)}$. Using this and that  $p^i\leq p^r=n_1/p$, we get
$$
 \frac{p^{i/2}}{ \hh_i}\geq
  \frac{p^{i/2}}{2\sqrt{2p}\,  p^{ i\, (2+ \alpha_d)}}\geq
    \frac{p^{r/2}}{2\sqrt{2p}\,  p^{ r\, (2+ \alpha_d)}}\geq
     \frac{p}{2 \sqrt{2}}\, \frac{\sqrt{n_1}}{n_1^{2 +\alpha_d}},
$$
which implies the result.
\end{proof}

\subsection{Nets for  steep vectors from $\st_1\cup \st_2$}

\bigskip

For the rest of steep vectors (i.e., for vectors from $\st_1 \cup \st_2$) we will
use the union bound together with a covering argument.
We first construct nets for ``normalized'' versions of the sets $\st_i$ and then provide
individual probability bounds for elements of the nets. The natural normalization would be
$x_{n_1}^{*}=1$, which we use for $\st_1$.
However, for individual probability bounds below and to have the same level of approximation, it is
more convenient to use a slightly different normalization for $\st_2$. Moreover, since $\st_2$ has a
constant jump, we can't just ignore the tail of the sequence as we will do for vectors in $\st _1$.
To overcome this difficulty, and to have  a  better control on the size of a net, we intersect this set
with the set of almost constant vectors. We set
$$
 \st'_1=\{x\in \st_1\,:\, x_{n_1}^{*}=1\} \quad \mbox{ and } \quad
 \st'_2=\st'_2(\rho)=\{x\in \st_2\,:\, x_{n_2}^{*}=1\}
 \cap \BB(\rho),
$$
where $0<\rho \leq 1/(d^{3/2}\, \bb)$.

\begin{lemma}[Cardinalities of nets]
\label{l:net}
Let $d\leq n$ be large enough and $0<\rho \leq 1/(d^{3/2}\, \bb)$.
Then, for each $i=1,2$,
there exists a $d^{-3/2}$-net $\Net_i$ in $\C^n$ for $\st_i'$  in $\ell_{\infty}$-metric with
$$
  |\Net_i| \le \exp\left( d n_i/4 \r),
$$
and for every $y\in \Net_i$ one has $y_j^*\leq 1/4 + 1/d^{3/2}$ for all $j\geq n_{i+1}$.
\end{lemma}

\begin{proof}
The constructions for $i=1$ and $i=2$ are quite similar, and we carry out the argument simultaneously for both
cases, making adjustments where necessary.
For every $x\in \st_i'$ ($i=1,2$) fix a permutation $\sigma=\sigma_x$ of $[n]$ such that $x_j^*=|x_{\sigma(j)}|$ for $j\le n$.

\smallskip

The main idea is to split a given vector from $\st_i'$ into three parts according to the behaviour of its coordinates
(essentially, parts corresponding to the largest coordinates, middle sized coordinates, and the smallest coordinates
with small adjustment in the case $i=2$) and approximate each part separately. Then we construct nets for vectors with the
same splitting and take the union over all nets. To be more precise, for each $x\in \st_i'$ ($i=1,2$) we consider a partition
of $[n]$ into three sets $B_1(x)$, $B_2(x)$, $B_3(x)$
corresponding to $x$, as follows. If $n_1=1$ (i.e., if $d^2/\log d \geq \a n$) we set $B_1(x)=\emptyset$.
Otherwise, if $n_1>1$, we set $B_1(x)=\sigma_x([n_1])$.
Further, we define sets $B_2(x),B_3(x)$ (this definition will depend on $i$).
For $i=1$ we set
$$
 B_2(x)= \sigma_x([n_2])\setminus B_1(x) \quad \mbox{ and }  \quad B_3(x)= \sigma_x([n]\setminus [n_2]).
$$
If $i=2$ then since $x\in  \BB(\rho)$ there exists $\lam _0 (x)$ such that the cardinality of the set
$$
 B_0(x) :=\{ j\leq n \, : \,  |x_j-\lam _0(x)|\leq \rho\|x\|_2 \}
$$
is larger than $n-n_3$. Note that by the assumption on $\rho$ and by Lemma~\ref{l:norma}, for every
$x\in \st_2'$ we have
$$
 \rho\|x\|_2\leq x^*_{n_2}/(4d^{3/2})= 1/(4d^{3/2}).
$$
Since $n_3<n/2$, there exists $j_0\geq n_3$ such that $\sigma_x(j_0)\in B_0(x)$.
Therefore,
$$
   |\lam _0(x)|\leq  x_{j_0}^* + \rho\|x\|_2 < x_{n_2}^*/4+ 1/(4d^{3/2})\leq 1/3.
$$
Using again that $x_{n_2}^*=1$ we observe that $\sigma_x (j) \notin B_0(x)$ for every $j\leq n_2$,
in particular, $B_1(x)\cap B_0(x)=\emptyset$.
Finally, in the case $i=2$, we choose an arbitrary subset $B_3(x) \subset B_0(x)$ of cardinality
$n-n_3$ and fix it, and we let $B_2(x)=  [n]\setminus (B_1(x)\cup B_3(x))$.

\smallskip

Note that if $n_1>1$ then for every $x\in \st_i'$ in both cases $i=1$ and $i=2$ we have
$$
  |B_1(x)|=n_1, \quad  |B_2(x)|=n_{i+1}-n_1, \quad \mbox{ and } \quad  |B_3(x)|=n-n_{i+1}.
$$
Thus, given a partition of $[n]$ into three sets
$B_1$, $B_2$, $B_3$ with cardinalities $|B_1|=n_1$, $|B_2|=n_{i+1}-n_1$, $|B_3|=n-n_{i+1}$,
it is enough to construct a net for vectors $x\in \st_i'$ with
$B_1(x)=B_1$, $B_2(x)=B_2$, $B_3(x)=B_3$ and then take the union of nets over all such partitions $\{B_1,B_2,B_3\}$
of $[n]$. In what follows, we skip the case $n_1=1$ (and $B_1=\emptyset$) as the simplest one, and assume that $n_1>1$.

\smallskip

Now we describe our construction. Note  that for $x\in \st_i'$ ($i=1,2$) we have $x_{n_1}^*\leq d^{(3/2)(i-1)}$
and, since $x\in(\st_0)^c$, we also have
\begin{equation}\label{decr2}
  x_1^*\leq (4d) x^*_{p}\leq (4d)^2 x^*_{p^2}\leq \ldots \leq (4d)^{r+1} x_{p^{r+1}}^* = (4d)^{r+1} x_{n_{1}}^*\leq d^{(3/2)(i-1)} \, (4d)^{r+1}
\end{equation}
(with corresponding adjustment for the case $n_1<p$).
Recall that we deal with the case $n_1>1$ (otherwise, $B_1(x)=\emptyset$ and we skip the first part).
Fix $I_0\subset [n]$ with $|I_0|=n_1$ (which will play the role of $B_1$). We  construct
a $d^{-3/2}$-net $\Net_{I_0}$ in the set
$$
   \st _{I_0}:=\big\{x\in(\st_0)^c:\,\sigma_x([n_1])=B_1(x) = I_0,\,\,x^*_{n_1}\leq d^{(3/2)(i-1)},\,\,x^*_{n_1+1}=0\big\}.
$$
Clearly, the nets $\Net_{I_0}$ for various $I_0$'s can be related by appropriate permutations, so
without loss of generality we can assume that $I_0=[n_1]$.
First, we construct a partition of $I_0$.
If $n_1=n_0\leq p$, let $I_1=[n_1]$. Otherwise, recall that $n_1=p^{r+1}$ and let
$$
  I_1=[p],\, \, \, I_2=[p^2]\setminus [p], \, \, \,  I_3=[p^3]\setminus [p^2], \, \ldots, \,\, \, I_{r+1}=[p^{r+1}]\setminus [p^{r}].
$$
Then the sets $I_1, \ldots, I_{r+1}$ form a partition of $I_0=[n_1]$.
Now, consider the set
$$\st^*:=\big\{x\in\st_{[n_1]}:\,\sigma_x(I_j)=I_j,\;\;j=1,2,\dots,r+1\big\}$$
and construct a $d^{-3/2}$-net $\Net^*$ in $\st^*$ in the following way.
Below we provide the proof for the case $n_1> p$ (i.e.,
when we have at least two sets in the partition), the other case is simpler.
By \eqref{decr2}, for every $x\in \st^*$, one has  $\|P_{I_j}x\|_\infty\le b:= d^{(3/2)(i-1)}\, (4d)^{r+2-j}$ for every
$j\le r+1$  (where $P_I$ denotes the coordinate projection onto $\C^I$). Set
$$
  \Net^*:=\Net_{1}\oplus\Net_{2}\oplus\cdots\oplus\Net_{r+1},
$$
where $\Net_{j}$ is a $d^{-3/2}$-net (in the $\ell_\infty$-metric) of cardinality at most
$$
  (3 b d^{3/2})^{2|I_j|} \leq  (4d)^{2(r+5-j) p^j}
$$
in the coordinate projection of the complex cube $P_{I_{j}}(bB_\infty^n)$.
Since
$d$ is large enough and $n_1=p^{r+1}$, we observe
$$
  \sum_{j=1}^{r+1}  2(r+5-j) p^j = 2  p^{r+1}\, \sum_{m=0}^{r} (m+4) p^{-m} \leq 10 p^{r+1} = 10 n_1,
$$
which implies
$$
  |\Net^*|\le\prod_{j=1}^{r+1}|\Net_{j}| \le\exp( 10 n_1  \log (4d)).
$$
To pass from the net for $\st^*$ to the net for $\st_{[n_1]}$, let $\Net_{[n_1]}$ be the union
of nets constructed as $\Net^*$ but for arbitrary partition $I_1'$, ..., $I_{r+1}'$ of $[n_1]$ with $|I_j'|=|I_j|$.
Using that $p= \lfloor (1/5)\sqrt{d/\log d} \rfloor$, we observe that
$$
 \sum _{j=1}^{r} {p^j}\log (ep) \leq \frac{p^{r+1}}{p-1} \log (ep) \leq  n_1 \log^2 d/\sqrt{d}.
$$
 Therefore, for large enough $d$,
$$
  |\Net_{[n_1]}| \le  |\Net^*|\,   \prod_{j=1}^{r} { {p^{j+1}} \choose {p^j }}
  \le  |\Net^*|\, \prod_{j=1}^{r} \left(ep \right) ^{p^j}
  \le \exp( 11 n_1  \log d  ).
$$

 Now we construct a net for the second part of the vector.
 Fix   $J_0\subset [n]$ with $|J_0|=n_{i+1} - n_1$ (which will play the role of $B_2$). We  construct
a $d^{-3/2}$-net  in the set
$$
   \st _{J_0}:=\{ P_{B_2(x)}x \, : \,  x\in\st_i',\,\, B_2(x)=J_0,\,\,x^*_{n_{i+1}}=0  \}.
$$
  Since $x_{n_1}^*\leq d^{(3/2)(i-1)}$ for $x\in \st_i'$, it is enough to take $d^{-3/2}$-net  ${\mathcal{K}}_{J_0}$ of cardinality at most
$$
   (3 d^{3/2}  d^{(3/2)(i-1)}) ^{2 |J_0|} \leq (3 d) ^{3i n_{i+1}}
$$
in the coordinate projection of the complex cube   $P_{J_{0}}(d^{(3/2)(i-1)} B_\infty^n)$.

 It remains to construct a net for the third part of the vector, corresponding to coordinates in $B_3$. Fix
 $B$ of cardinality $n-n_{i+1}$ and consider the set
$$
   \st _{B}:=\{P_{B_3(x)}x \, : \,  x\in \st_i',\,\, B_3(x)=B\}.
$$
 If $i=1$ then, by definitions, $\|y\|_{\infty}< d^{-3/2}$ for every $y\in \st _{B}$, therefore our net, $\mathcal{O}_B$, will consist of $0$ only.
 In the case $i=2$, for $x\in \st_2'$ and $j\in B$, using Lemma~\ref{l:norma} and the condition on $\rho$,
 we have that
$$
 |x_j-\lam _0(x)|\leq \rho\|x\|_2  \leq  (\rho \bb/4) x_{n_2}^*\leq 1/(4 d^{3/2})
  \quad \mbox{ and } \quad |\lam _0(x)|\leq   1/3.
$$
Take a $3/(4 d^{3/2})$-net $\mathcal{O}$ in the set $\{\lam \in \C\, :\, |\lam |\leq 1/3\}$ of cardinality at most
$2 d^3$ and let
$$
   \mathcal{O}_B:=\{y\in \C^B\, :\, \exists \lam \in \mathcal{O} \, \, \mbox{such that} \, \, \forall j\in B \, \, \mbox{one has}
   \, \, y_j = \lam\} .
$$
 Clearly, $\mathcal{O}_B$ is a $d^{-3/2}$-net for $\st _{B}$.

Finally consider the net
$$
 \Net:=\bigcup\{y=y_1+y_2+y_3:\,y_1\in\Net_{I_0},\,y_2\in\mathcal{K}_{J_0},\, y_3\in \mathcal{O}_B\},
$$
where the union is taken over all partitions of $[n]$ into $I_0, J_0, B$ with
$|I_0|=n_1$, $|J_0|=n_{i+1}-n_1$, and $|B|=n-n_i$.
Clearly, $\Net$ is a $d^{-3/2}$-net for $\st _i'$ and, using (\ref{new-par-nn}) and (\ref{new-par-aa}),
we obtain for large enough $d$,
$$
 |\Net|\le {n\choose {n_{i+1}}}\, {n_{i+1}\choose n_1}\, |\Net_{I_0}|\, |\mathcal{K}_{J_0}|\, |\mathcal{O}_B|
 \le \left(\frac{e n}{n_{i+1}}\r)^{n_{i+1}}\, \left(\frac{e  n_{i+1}}{n_1}\r)^{n_{1}}
 \,  (3 d)^{11 n_1+ 3 i n_{i+1}+ 3}
$$
$$
   \leq \exp\left( 7 n_{i+1} \log d \r) \leq \exp\left( 7 (a_{i+1}/a_i) d n_i  \r)
   \leq \exp\left( d n_{i}  /4 \r).
$$
Without loss of generality (by removing unnecessary vectors from $\Net$), we may assume that  every $y\in \Net$
approximates some $x\in \st_i'$. This implies that for every $y\in \Net$ one has $y_j^*\leq 1/4+1/d^{3/2}$
for all $j\geq n_{i+1}$,
 completing the proof.
\end{proof}

\subsection{Individual probability bounds}

To obtain the lower bounds on $\|(M+\WW)x\|_2$, where $\WW$ is a fixed matrix, for vectors $x$ from our nets, we  investigate the
behavior of coordinates of $(M+\WW)x$, that is of the inner products $\la R_i(M+\WW), \bar x  \ra$. One of the tools
that we use is  Theorem~\ref{graph th known} together with Lemma~\ref{c:SJ} applied to the $2m$ columns of $M$ corresponding
to the $m$ biggest and $m$ smallest (in the absolute value) coordinates of $x$ with properly
chosen $m$. Then, using jumps,  we show that the inner product of some row $R_i(M+\WW)$ with the first
part of the vector and with the second part of the vector cannot be simultaneously large.
This will reduce the set of matrices under consideration to a much smaller set, where it is
easier to obtain a good probability bound. To make our scheme work we will use the following
subdivision of $\Mc$.

Given $J\subset [n]$ and $M\in \Mc$ we denote
$$
  I (J, M) = \{i \leq n \, : \, |\supp R_i(M) \cap J |       =1\}
$$
({\it cf.},  the definition of $\il(M)$,  $\ir(M)$   before  Lemma~\ref{c:SJ}, clearly, if
we split $J$ into $J^\ell$ and $J^r$, then $I (J, M)= \il(M) \cup \ir(M)$).

Fix $J\subset [n]$.
Given a subset $I$ of $[n]$ and $V=\{v_{ij}\}\in \Mc$, consider the class
$$
   \f (I, V) =  \left\{M\in \Mc \, :\,   \quad I (J, M) =  I \quad
     \mbox{ and } \quad
      \, \, \forall i\leq n\, \forall j\in J^c \, \,\, \,
     \mu_{ij}= v_{ij} \right\}
$$
(depending on the choice of $I$ such a class can be empty). In words, we first fix the columns
indexed by $J^c$ and then fix the set of indices $I$ such that the rows indexed by $I$  have only one 1 in columns indexed by $J$.
Clearly, $\Mc$ splits into  disjoint union of classes $\f (I, V)$ over some subset of matrices  $V$ in $\Mc$ and
all $I\subset [n]$.

\begin{lemma}[Individual probability]
\label{individual}
There exist absolute constants $C>1>\eps >0$ such that the following holds.
Let $C<d< n$, $i=1,2$, and $\WW$ be a complex $n\times n$ matrix.
 Assume $x\in \C^n$ satisfies
$$
   x^*_{n_i}\geq 1/2 + x^*_j  \, \, \, \mbox{ for every }\, \, j\geq  n_{i+1}.
$$
Denote
$
   E(x) :=\left\{ M \in \Mc\, :\, \|(M+\WW)x  \|_2\leq  \sqrt{ n_i d} /24\r\}.
$
Then
$$
 \Prob(E\cap  \Omega_{2n_i, \eps} )\leq \exp(- n_i d/2).
$$
\end{lemma}

\begin{proof}
Fix $x$ satisfying the condition of the lemma. Let $\sigma$ be a permutation  of $[n]$ such that
$x_j^*=|x_{\sigma(j)}|$ for all $j\le n$. Denote $m=n_i$. Let
$$
  J^{\ell} =\sigma ([n_i]) \quad \mbox{ and } \quad   J^{r}=\sigma ([n-n_i+1, n]).
$$
 Denote $J= J^{\ell} \cup J^{r}$.
 Fix $\eps>0$ small enough. We assume that
 $\aa  < c_{\ref{graph th known}} \eps /2$. Then
 $m=n_i\leq n_2 \leq c_{\ref{graph th known}} \eps n/2d$.

Let $M\in\omep$.  Let the sets $\il(M)$ and  $\ir(M)$ be defined as
 before  Lemma~\ref{c:SJ}. Since $|J|=2m \le c_{\ref{graph th known}} \eps n/d $, this lemma implies  that
$
 |\il (M)|,\,|\ir (M)|\in[(1-4\eps)md,\,md],
$
 in particular $I=\il(M)\cup \ir(M)$ satisfy
 \begin{equation}\label{cond-card}
 |I|\in[2(1-4\eps)md,\,2md].
 \end{equation}
Now we split $\Mc$ into disjoint union of classes $\f (I, V)$ defined at the beginning of this subsection
 and note that $\omep \cap \f (I, V)\ne \emptyset$
implies that
$I$ satisfies (\ref{cond-card}).
Thus, to prove our lemma it is enough to prove uniform upper bound for  such classes,
indeed,
$$
  \Prob(E(x)\cap  \omep )\leq \max  \p (E(x) \cap \omep\,|\, \f (I, V))
  \leq \max  \p (E(x) |\, \f (I, V)),
$$
where the first maximum is taken over all classes $\f (I, V)$ with $\omep \cap \f (I, V)\ne \emptyset$ and
the second maximum is taking over $\f (I, V)$ with $I$'s satisfying (\ref{cond-card}).

Fix such a class $\f (I, V)$ for some
$I\subset [n]$ with $t_1:=|I|\in [2(1-4\eps)m d,\, 2md]$
and denote the uniform probability on it
just by $\p_\f$, that is
$$
 \p_\f (\cdot) = \p( \cdot \, | \,  \f (I, V)).
$$
Without loss of generality we assume that $I=[t_1]$.

By definition, for matrices $M\in E(x)$ we have
$$
  \|(M+\WW)x\|_2^2 = \sum _{i=1}^n | \langle R_{i}(M+\WW),\, \bar x \rangle|^2 \le md/576.
$$
Therefore there are at most $t_0:= md/36$
rows $R_i=R_i(M+\WW)$ with $| \langle R_{i},\, \bar x \rangle|\ge 1/4$. Hence,
$$
 |\{i\in I\, : \,| \langle R_{i},\, \bar x \rangle|< 1/4\}|\ge
   t_1-t_0.
$$
  Denote $t:=\lceil  t_1 - t_0\rceil$. The above bound
implies that for every $M\in  E(x)$ there is a set of indices $B(M)\subset I$
such that $|B(M)|=t$ and  for every $i\in B(M)$ one has
$| \langle R_{i},\, \bar x \rangle| < 1/4$.
Thus, denoting
$$
 \Omega_i:=\{M\in\f (I, V):\, |\langle R_{i},\, \bar x \rangle|< 1/4 \},
$$
we obtain
\begin{align}\label{bigprobf}
  \Prob_{\f}( E(x))&\le \sum _{B\subset I \atop |B|=t} \, \Prob_{\f}\Big(\bigcap_{i\in B}\Omega_i\Big)
  \leq   {t_1 \choose t}\, \max _{B\subset I \atop |B|=t} \, \Prob_{\f}\Big(\bigcap_{i\in B}\Omega_i\Big)
  \leq \left(\frac{e t_1}{t_0}\right)^{t_0}  \,\max _{B\subset I \atop |B|=t} \, \Prob_{\f}\Big(\bigcap_{i\in B}\Omega_i\Big).
\end{align}
Next for every $i\in I$ by $\f_{I}^\ell(i)$ and $\f_{I}^r(i)$ denote the sets
$$
  \{M\in\f(I, V):\, i\in\il (M)\} =
  \{M\in\f(I, V):\,|\supp R_i(M)\cap J^\ell|=1, \,
  \supp R_i(M)\cap J^r=\emptyset\}
$$
and
$$
 \{M\in\f(I, V):\, i\in\ir (M)\}=
  \{M\in\f(I, V):\,|\supp R_i(M)\cap J^r|=1, \,
  \supp R_i(M)\cap J^\ell=\emptyset\}.
$$
Clearly, for every $i$, the sets $\f_{I}^\ell(i)$ and $\f_{I}^r(i)$ form a partition $\f(I, V)$.
We show  that for every $i\in I$ either
$\Omega_i\subset \f^\ell_I(i)$ or $\Omega_i\subset \f^r_I(i)$.
Indeed, assume that $M_1\in\f^\ell_I(i)$ and $M_2\in\f^r_I(i)$.
By the definition of our sets and by the conditions on $x$, we have
$$
  J_1:=  \supp R_i(M_1)\setminus J=\supp R_i(M_2)\setminus J,
$$
and  there exist $j_\ell\in J^\ell$, $j_r\in J^r$ such that
\begin{align*}
 &\langle R_{i}(M_1),\, \bar x \rangle= x_{j_\ell}+\sum_{j\in J_1}  x_j \quad \mbox{ and } \quad
\langle R_{i}(M_2),\, \bar x \rangle=x_{j_r} + \sum_{j\in J_1}x_j.
\end{align*}
Hence,
$$
|\langle R_{i}(M_1+\WW),\, \bar x \rangle| + |\langle R_{i}(M_2+\WW),\, \bar x \rangle|\ge
 |\langle R_{i}(M_1+\WW),\, \bar x \rangle-\langle R_{i}(M_2+\WW),\, \bar x \rangle|
 $$
 $$
  = |x_{j_\ell}-x_{j_r}| \geq   x^*_{n_i} - |x_{j_r}|  \geq  1/2.
$$
Thus, it is impossible to simultaneously have both
$$
  |\langle R_{i}(M_1+\WW),\, \bar x \rangle|< 1/4 \quad \mbox{ and } \quad
  |\langle R_{i}(M_2+\WW),\, \bar x \rangle|< 1/4
$$
and therefore either
$\Omega_i\subset \f^\ell_I(i)$ or $\Omega_i\subset \f^r_I(i)$.
 This implies for every $B\subset I$ with $|B|=t$,
$$
   \Prob\Big(\bigcap_{i\in B}\Omega_i\Big) \leq
   \max_{B_0\subset B}  \, \Prob_{\f}\Big(\bigcap_{i\in B_0}\f^\ell_I(i) \,
   \bigcap  \, \bigcap_{i\in B\setminus B_0}  \f^r_I(i) \Big) =
  \max_{B_0\subset [t]}  \, \Prob_{\f}\Big(\bigcap_{i\in B_0}\f^\ell_I(i) \,
   \bigcap  \, \bigcap_{i\in [t]\setminus B_0}  \f^r_I(i) \Big) ,
$$
where in the last equality we used permutation invariance.

\begin{claim} \label{cl-one} If $d$ is large enough and $\eps$ is small enough then
 for every $B_0\subset [t]$ one has
$$
   \Prob_{\f}\Big(\bigcap_{i\in B_0}\f^\ell_I(i) \,
   \bigcap  \, \bigcap_{i\in [t]\setminus B_0}  \f^r_I(i) \Big)
   \leq e^{-t/3}.
$$
\end{claim}

Recall that $t_1\in [2(1-4\eps)m d,\, 2md]$, $t_0= md/36$, and $t=\lceil  t_1 - t_0\rceil$, so that
$$
  t/3-t_0\log(et_1/t_0)\geq md\left((2-8\eps-1/36)/3- (1/36) \log(72 e)) \right) \geq md/2,
$$
provided that $\eps$ is small enough. Therefore Claim~\ref{cl-one} and (\ref{bigprobf}) imply
the desired result.
\end{proof}

\medskip

\begin{proof}[Proof of Claim~\ref{cl-one}]
Fix $B_0\subset[t]$. Denote $\lll :=|B_0|$ and without loss of generality
assume that  $\lll \geq t/2$. Let $q=\lfloor \lll /2\rfloor$.
To compare the cardinalities of
$$
  A: = \bigcap_{i\in B_0}\f^\ell_I(i) \,  \bigcap  \, \bigcap_{i\in [t]\setminus B_0}  \f^r_I(i)
$$
and $\f(I, V)$ we construct a relation $R$ between them as follows.
Let $M\in A$. We say that $(M, M')\in R$ if  $M'\in \f(I, V)$ can be obtained from $M$
in the following way.  Choose a subset $B_1\subset B_0$ of cardinality $q$.
There are
$$
 {\lll  \choose q} \geq \frac{2^{\lll }}{2\sqrt{\lll }}
$$
such choices.  Let $i_1<i_2<\ldots <i_{q}$ be the elements of $B_1$.
Recall that $M\in \f^\ell_I(i_s)$ for every $s\leq q$.
Let $j_1, \ldots, j_q$ be elements of $J^\ell$ such that
$M$ has ones on positions $(i_s, j_s)$ for $s\leq q$.
Choose a subset $B_2\subset \ir(M)$ of cardinality $q$.
There are
$$
 {|\ir(M)| \choose q}\geq {\lceil(1-4\eps)md \rceil \choose q}
$$
  such choices.
Let $v_1<v_2<\ldots <v_{q}$ be elements of $B_2$. Let $w_1, \ldots, w_q$ be elements of $J^r$ such that
$M$ has ones on positions $(v_s, w_s)$ for $s\leq q$. Let $M'\in \f(I, V)$ be obtained from $M$ by
substituting ones with zeros on places $(i_s, j_s)$ and $(v_s, w_s)$ and substituting zeros with ones
on places $(i_s, w_s)$ and $(v_s, j_s)$ for all $s\leq q$. By construction we have
$$
   |R(A)| \geq \frac{2^{\lll }}{2\sqrt{\lll }}\, {\lceil (1-4\eps)md\rceil \choose q}.
$$

Now we estimate the cardinalities of preimages. Let $M'\in R(A)$. Then the set
$B_3=B_0\cap \ir(M')$ must have cardinality $q$. Write $B_3=\{i_1, i_2, \ldots i_{q}\}$
with $i_1<i_2<\ldots <i_{q}$. Let $w_1, \ldots, w_q$ be elements of
$J^r$ such that
$M'$ has ones on positions $(i_s, w_s)$ for $s\leq q$. If $(M, M')\in R$, $M$
has to have zeros on those positions. We now compute how many such matrices
$M\in \f(I, V)$ can be constructed, that is, how many possibilities to have ones
in rows $i_s$, $s\leq q$, exist. Since $M'\in R(A)$, we have
$$
  |\il(M') \setminus B_0|= |\il(M')| - (|B_0|-q) \leq md.
$$
Choose
$B_4\subset \il(M') \setminus B_0$ of cardinality $q$.
Write $B_4=\{v_1, v_2, \ldots, v_{q}\}$ with $v_1<v_2<\ldots <v_{q}$.
Let $j_1, \ldots, j_q$ be elements of $J^r$ such that
$M'$ has ones on positions $(v_s, j_s)$ for $s\leq q$.
Then $M$ is obtained from $M'$ by substituting zeros with ones
on places $(i_s, j_s)$ and $(v_s, w_s)$ and substituting ones with
zeros on places $(i_s, w_s)$ and $(v_s, j_s)$ for all $s\leq q$.
Thus, $|R^{-1}|$ is bounded above by the number of choices for the set $B_4$,
that is
$
   |R^{-1}(A)| \leq  { md \choose q}.
$
Using that for every integers $N$ and $s$ with $N-s>q$ one has
$$
 \frac{ {N\choose q}}{ {N-s\choose q}}=
   \frac{N...(N-s+1)}{(N-s)...(N-s-q+1)}
   \leq \left( \frac{N-s+1}{N-s-q+1}\right)^s\leq \exp
   \left(\frac{sq}{N-s-q+1}\right),
$$
that $q=\lfloor \lll /2\rfloor \leq t/2$, $t\leq t_1 -t_0 \leq  (2-1/36) md$,
and Claim~\ref{f} we observe that
$$
  \frac{|A|}{|\f(I, V)|}\leq
\frac{2\sqrt{\lll }}{2^{\lll }}\,
  \exp\left(\frac{q \, 4\eps md }{(1-4\eps) md - q +1}\right)
  \leq \frac{\sqrt{2 t}}{2^{t/2}}\,
    \exp\left(\frac{2\eps t md }{(1-4\eps) md - t/2}\right)
$$
$$
  \leq
  \frac{\sqrt{2 t}}{2^{t/2}}\, \exp\left(\frac{144\eps t}{1-288\eps}\right)
  \leq e^{-t/3},
$$
provided that $\eps$ is small enough and $d$ (hence $t$) is large enough.
\end{proof}

\subsection{Proof of Theorem \ref{t:steep}}

We are ready to complete the proof.

\begin{proof}[Proof of Theorem \ref{t:steep}.]
Recall that $d$ is large enough,  $\eps _0=\sqrt{(\log d) /d}$, $p=\lfloor 1/5\eps _0\rfloor$, and let $\eps$ be
a small positive constant from Lemma~\ref{individual}.
In most formulas below we assume that $n_0>1$, otherwise $T_0=\emptyset$ and the proof is easier. We make
corresponding remarks in the  text. Below we deal with matrices from
$$
  \Omega_0 =
  \bigcap _{j= 2}^{r+1}\Omega_{p^j,\eps_0}  \,  \cap\, \Omega_{k_1,\eps_0}
 \,  \cap\,   \bigcap _{i= 1}^{2}\, \Omega_{2 n_i,\eps },
$$
where $k_1=\min\{n_0, p\}$ and where
we do not have the first intersection if $n_1=n_0\leq p$ and we do not have the second term if
$n_1=n_0=1$.

\smallskip

If $x\in \st_0$ and $M\in \Omega_0$ then Lemma~\ref{l:T0}
implies that
$$
 \|(M-\ww  \idmat)x\|_2 \geq \min\left\{\sqrt{d/8n},\, \frac{p\sqrt{n_1 d}}{8 \hh_{r+1}}\r\}\, \|x\|_2.
$$

We turn now to the case $x\in \st_i$ for $i=1,2$. Let
\begin{align}
 &\Event_{i}:=\Big\{M\in\Mc\,:\,\exists\;x\in \st_i\,\, \,\mbox{such that}\,\,\,\|(M-\ww  \idmat) x\|_2\le
 \frac{\sqrt{ n_i d}}{25 \, b_i }\, \|x\|_2\Big\},\notag
\end{align}
where $b_1=\hh_{r+1}$ and $b_2=d^{3/2} h_{r+1}$ in the case $n_0>1$ and $b_2=d\sqrt{n}$ in the case $n_0=1$.
By Lemma~\ref{l:norma} for $x \in \st_i$ one has $\|x\|_2\le b_i  x_{n_i}^{*}$.
Thus, for $M\in \Event_{i}$ there exists $x=x(M)\in \st_i$ with
$$
 \|(M-\ww  \idmat) x\|_2\le \frac{\sqrt{ n_i d}}{25 }\,  x_{n_{i}}^{*}.
 $$
Normalizing $x\in \st_i$, so that $x_{n_{i}}^{*}=1$ (that is, $x\in \st_i'$), we observe that
there exists $y=y(x)$ from the net constructed in Lemma~\ref{l:net}  with $y_{n_{i}}^{*}\geq 1-d^{-3/2}>3/4$,
and $y_{j}^{*}\leq 1/4$ for $j> n_{i+1}$ and such that
$$
 \|x-y\|_2\leq \sqrt{n}\,  \|x-y\|_\infty \leq  d^{-3/2}\sqrt{n} \leq \frac{1}{600}\, \sqrt{n_i/d}.
$$
Therefore, using that $\|M\|=d$ and $|\ww|\leq d$, we have
$$
   \|(M-\ww  \idmat) y\|_2\le \|(M-\ww  \idmat) x\|_2 + (\|M\| +|\ww|) \|x-y\|_2 \leq
 \sqrt{ n_i d} /24 .
$$
Now we use the union bound over vectors in the net together with individual probability bounds.
Lemmas~\ref{individual} and \ref{l:net} imply for $i=1,2$,
$$
 \p\left(\Event_{i} \cap \Omega _0\r) \leq \exp \left(-n_i d/4\r).
$$

Combining all cases  we obtain that for $x\in \st$ one has $\|(M-\ww  \idmat) x\|_2 \leq A \nx$, where
$$
 A:= \min \left(  \frac{  \sqrt{d} }{ 2\sqrt{2n} },\,  \frac{  \sqrt{ n_1 d} }{25 b_1 }, \,  \frac{ \sqrt{ n_2 d} }{25 b_2 }\, \r)\, ,
$$
with probability at most $p_0:=\p\left( \Omega_0^c\r)+ \exp \left(-n_1 d/4\r) + \exp \left(-n_2 d/4\r)$.

 We first estimate $A$. If $n_0=n_1=1$ then $\st _0=\emptyset$, $d^2>n$, and $h_{r+1}=\sqrt{n}$. Therefore
$$
 \frac{\sqrt{n_1 d}}{25 b_1}= \frac{\sqrt{d}}{25\sqrt{n}} \quad  \mbox{ and } \quad
 \frac{\sqrt{n_2 d}}{25 b_2}\geq  \frac{\sqrt{\aa  }}{30d},
$$
which implies that $A\geq c/d$ in this case. If $n_1>1$ then
$n_1 d \geq \a n \log d / d \geq \aa n / d\approx n_2$.
Therefore, in the case $1<n_0=n_1\leq p$, one has
$$
     A\geq \sqrt{a_2 n}/(26 d^{3/2}h_{r+1}) \geq \sqrt{a_2 n}/(40 d^{3}\sqrt{\log d}) ,
$$
while in the case $n_0>p$, using that by (\ref{new-par-nnnn}), $n_1\leq a_1 \sqrt{\log d}\,  n/5d^{3/2}$,
$$
     A\geq \frac{\sqrt{a_2 n}}{26 d^{3/2}h_{r+1}} \geq \frac{\sqrt{a_2 n}}{26 \sqrt{3p}\, d^{3/2} n_1^{2+\alpha_d}}
     \geq \frac{\sqrt{a_2 n }\, d^{3+3\alpha d/2}}{3 \sqrt{p}\, d^{3/2} a_1^3 n^{2+\alpha_d}}  \geq
     \frac{\sqrt{a_2}\, d^{5/4} \log^2 d}{3  a_1^3 n^{3/2+\alpha_d}}.
$$

We now estimate the probability $p_0$ using Theorem~\ref{graph th known}. Recall that $c_1$, $c_2$, ... always
denote (sufficiently small) positive absolute constants.
First note that Theorem~\ref{graph th known} implies
$$
   p_1:= \sum_{i=1}^2 \left( \p\left( \Omega_{2 n_i,\eps }^c\r)  + \exp \left(-n_i d/4\r) \r)
$$
$$
  \leq
   \sum  _{i=1}^{2}  \left( \exp\left(-\frac{\eps^2 d n_i }{4} \log\left(\frac{e c_{\ref{graph th known}} \eps  n }{2 d n_i}\right)\right)
    + \exp \left(-\frac{n_i d}{4}\r) \r) \leq \exp \left(-c_1 n_1 d\r).
$$
In the case $n_1=n_0=1$ we have  $\a n \leq d^2/\log d$ and hence
$p_1\leq  \exp \left(-c_2 \sqrt{n} \r)$.
In the case $n_1>1$ we have $\a n \log d\geq d^2$, hence
$$
   n_1d\geq n_0d\geq (\a n\log d)/d\geq  \sqrt{\a n\log d},
$$
thus  again
$p_1\leq  \exp \left(-c_2 \sqrt{n} \r)$.

In the case $1<n_0=n_1\leq p$ we have $k_1=n_1$, $\a n \log d\geq d^2$, and $\a n \log^{3/2} d \leq d^{2.5}$.
Therefore,  by Theorem~\ref{graph th known},
$$
  p_2:=\p\left(\Omega_{k_1,\eps_0}^c\r) \leq  \exp\left(-\frac{n_1 \log d }{8} \log\left(\frac{e c_{\ref{graph th known}}  n\log d }{d^{3/2} n_1}\right)\right)
  \leq   \exp \left(-c_3 \log ^2 n  \r).
$$
 Recall that in the definition of $\Omega_0$
we do not have the first intersection if $n_1=n_0\leq p$ and we do not have the second term if
$n_1=n_0=1$. This implies that in the case $n_1\leq p$ we have  $p_0\leq p_1+p_2\leq \exp\left(-c_4  \log ^2 n   \r).$

\smallskip

Finally,  in the case $n_1> p$, we have $k_1=p$, $r\geq 1$, and,
$c_4 n \geq d^{5/2}/\log ^{3/2} d$. Therefore,  by Theorem~\ref{graph th known},
$$
 p_3:= \sum  _{i=2}^{r+1}\p\left( \Omega_{p^j,\eps_0}^c\r) + \p\left(\Omega_{k_1,\eps_0}^c\r) \leq
  \sum  _{i=1}^{r+1}
  \exp\left(-\frac{p^i\log d }{8} \log\left(\frac{e c_{\ref{graph th known}} \eps _0 n }{d p^i}\right)\right)
$$
$$
\leq
  \exp\left(-\frac{p\log d }{9} \log\left(\frac{e c_{\ref{graph th known}} \eps _0 n }{d p}\right)\right)
  \leq \exp\left(-c_5 \sqrt{ d \log d} \, \log n
  \right).
$$
Since $p_0\leq p_1+p_2+ p_3$, the desired estimate follows.
\end{proof}






\section{Bounds for essentially non-constant vectors and completing the proof of the main theorem}
\label{s: quantitative}

In this section, we complete our proof of the lower bound for the smallest
singular value of a random matrix uniformly distributed in
$\Mc$, shifted by $\ww\,\idmat$ for a fixed $\ww\in\C$.
To better separate various techniques used in this paper,
we prefer to give an ``autonomous'' proof of the result, conditioned
on a rather general assumption about the structure of the kernel of our random matrix.
This assumption, for a specific choice of parameters, is actually proved in Section~\ref{steep}
(see Remark~\ref{cortwoth}),
so the argument presented here implies the main result of the paper regarding the magnitude of $s_{n}$.
We provide the details in Section~\ref{prmainth}.

We start by introducing notations. Fix an $n\times n$ (complex) matrix $\WW$.
Further, take positive parameters $\kappa, \rho\in (0,1)$, and $\delta\in (0, 1)$
(the parameters may and in fact {\it will} depend on $n$ and $d$, moreover, we take
$\delta$ very close to zero).
Define the subset  $\Inc(\rho,\delta)$ of the unit sphere in $\C^n$ by
\begin{equation*}
\Inc(\rho,\delta):=\Big\{x\in\C^n:\,\|x\|_2=1\quad \mbox{ and }\quad
  \forall\lambda\in\C\quad\vert \{i\le n :\, \vert x_i-\lambda\vert
> \rho\}\vert > \delta n\Big\}.
\end{equation*}
Note that for  $\delta =  \nn/n \approx  \aaa/\log d$ one has
$$
 \Inc(\rho,\delta) = \left(\C^n\setminus \BB (\rho)\r) \cap \left\{x\in\C^n:\,\|x\|_2=1\r\}.
$$
Further, define two events
\begin{align*}
\Event_{\ref{s: quantitative}}=
\Event_{\ref{s: quantitative}}(\WW,\kappa,\rho,\delta)&:=\bigl\{M\in\Mc:\,
\forall\;x\in\C^n \mbox{ with }\|x\|_2=1\mbox{ and}\\
&\hspace{1cm}\min(\Vert (M+\WW) x\Vert_2, \Vert \bar x (M+\WW)\Vert_2)
\le \kappa\mbox{ one has } x\in\Inc(\rho,\delta)\bigr\},
\end{align*}
and
$$
\Event_{\ref{th-ssv}}=
\Event_{\ref{th-ssv}}(\WW,\kappa)
:=\bigl\{M\in\Mc:\,s_{n}(M+\WW) \leq \kappa\bigr\}.
$$
The parameters $\WW,\rho,\delta,\kappa$ are usually clear from the context, and we will simply write
$\Event_{\ref{s: quantitative}}$ and $\Event_{\ref{th-ssv}}$ to denote the respective events.

\begin{theor}\label{th-ssv}
There exist positive absolute constants $c$, $C_0$, and $C$ with the following property.
Let $\delta\in (0,1)$, $\rho \in (0,1)$, $\kappa:=\rho^2/16$,  and
$$
  C \leq d\leq \frac{c \delta}{\log (e/\delta)} \, \, n.
$$
 Further, assume that $\WW$ is a complex matrix such that the event
$\Event_{\ref{s: quantitative}}=\Event_{\ref{s: quantitative}}(\WW,\kappa,\rho,\delta)$
has probability at least $1-1/n^2$. Then
$$
\Prob(\Event_{\ref{th-ssv}})\leq \frac{C_0 \sqrt{\log (e/\delta)}}{\delta^{3/2} }\, \, \frac{1}{\sqrt{d}}.
$$
\end{theor}

One can describe the structure of the above theorem as follows: provided that for a random matrix $M$
uniformly distributed in $\Mc$, vectors ``close'' to the kernel of $M$ are unstructured
(i.e., not almost constant),
the smallest singular value of $M$ is at least $\kappa$
with large probability (later we choose $\kappa$  to be a (negative) constant power of $n$).
Theorem~\ref{th-ssv} should be compared with the recent results of \cite{Cook-circ,BCZ}
discussed in the introduction.
The high-level structure of the theorem is in many respects similar to
\cite[Lemmas~6.2, 6.3]{Cook-circ},
where invertibility properties of the random matrix are also derived conditioned on a ``good'' event
which encapsulates properties of ``almost null'' vectors of the matrix.
In \cite{Cook-circ}, the linear spans of the matrix rows of $M$ are studied 
with the help of an auxiliary collection of random vectors (denoted as $u^{(i_1,i_2)}$), which are defined on a certain ``good'' event,
are measurable with respect to the sigma-algebra generated by the submatrix $M^{i_1i_2}$
and possess several specific structural properties (see \cite[Definition~6.1]{Cook-circ}).
Estimates of the smallest singular value are then reduced to bounding the inner product of $u^{(i_1,i_2)}$
with the difference of $i_1$-st and $i_2$-nd rows, for all pairs of indices $i_1,i_2$ \cite[Lemma~6.2]{Cook-circ}.
Existence of such random vectors $u^{(i_1,i_2)}$
is verified in \cite{Cook-circ} by considering the singular vectors of matrices $M$ corresponding to their smallest singular value.
This creates an additional level of abstraction, which we avoid in this paper by studying the singular vectors directly.
More specifically, given a class of matrices in $\Mc$ sharing the same $(n-2)\times n$ submatrix, we first choose a singular
vector corresponding to a matrix with a small $s_n$, then we use it for all matrices in the class to study invertibility
(see Lemma~\ref{l:H3}).

Our intention was to extract a linear algebraic part of the argument which is independent of the particular model
of randomness and to present it in a self-contained way (see Subsection~\ref{subs: srrsqm} below).
Estimates for the smallest singular value are connected to
distance estimates involving {\it pairs} of rows rather than the distance of a single
row to the span of the remaining rows (see \cite[Lemma 3.5]{RV}). We expect that those linear-algebraic arguments
may be used in other
random models of matrices with fixed row- or columns-sums.
The relations of Subsection~\ref{subs: srrsqm} are not explicitly given in \cite{Cook-circ},
although proofs of our lemmas use arguments similar to those empoyed in \cite[Lemma~6.2]{Cook-circ},
as well as in \cite{RV}.

\subsection{Some relations for random square matrices}\label{subs: srrsqm}

In this subsection we present two lemmas -- one probabilistic and the other linear algebraic --
which work for a wide class of square matrices.
The next lemma is analogous to \cite[Lemma 3.5]{RV}. The proof follows the same lines,
and we include it for the sake of completeness.

\begin{lemma}
\label{l:RV}
Fix parameters $\rho,\delta,\delta_0,\varepsilon>0$,
and assume that $0\leq \delta_0<\delta\leq 1-1/n$.
Further, let
$$K_0\subset K:= \{(i,j)\, :\, 1\leq i\ne j\leq n\}$$
be such that
$|K_0|\geq (1-\delta_0) n(n-1)$. Let $A$ be an $n\times n$ random matrix on
some probability space such that $\sum_{i=1}^n \row_i(A)=v$ a.s.\ for a fixed
vector $v\in \C^n$.
Then
\begin{align*}
\Prob\big\{&\inf_{x\in \Inc(\rho,\delta)} \Vert
\bar x  A\Vert_2 \leq  \varepsilon \rho\big\}
\\
&\leq \frac{1}{n^2(\delta-\delta_0)}\sum_{(i,j)\in K_0}
\Prob\big\{\d\big(\row_i(A), \spn\big\{\{\row_k(A)\}_{k\neq i,j},\, \row_i(A)+\row_j(A)\big\}\big)< \varepsilon \big\}.
\end{align*}
\end{lemma}
\begin{proof}
In this proof for $i\leq n$ we denote $\row_i(A)$ just by $\row_i$.
Without loss of generality we assume that $\sum_{i=1}^n \row_i=v$ everywhere on the probability space.
For each pair $(i,j)\in K$, set
$$
  d_{ij}=d_{ij}(A):=\d\big(R_i, \spn\big\{\{R_k\}_{k\neq i,j},\, R_i+R_j\big\}\big) .
$$
Note that $d_{ij}=  \d\big(R_i, \spn\big\{\{R_k\}_{k\neq i,j},\,v\big\}\big).$
Since $\bar x  A=\sum_{k=1}^n \bar{x}_k R_k$, for every $(i,j)\in K$ we have
\begin{align*}
\Vert \bar x  A\Vert_2= \Big\Vert (\bar{x}_i-\bar{x}_j)\row_i+\bar{x}_j v+\sum_{k\neq i,j}(\bar{x}_k-\bar{x}_j)\row_k\Big\Vert_2\geq
|\bar{x}_i-\bar{x}_j|d_{ij}.
\end{align*}
The above relation is the principal point of the proof. Now, if ``many'' distances $d_{ij}$
are ``large'', then, since the vector $x$ is essentially non-constant, we can find a pair $(i,j)$
such that both $|\bar{x}_i-\bar{x}_j|$ and $d_{ij}$ are large, and we get a lower bound $\Vert \bar x  A\Vert_2>\varepsilon\rho$.
Thus, we can estimate the probability of the considered event in terms of probability that
``not so many'' distances $d_{ij}$ are large which is in turn done via Markov's inequality.
Below is a rigorous argument.

Let $K_1:=\{(i,j)\in K_0:\, d_{ij}\geq\eps\}$. Denote by $\Event$ the event that $|K_1|>(1-\delta) n^2-n$.
Note that if $M\in \Event^c$, we have
$$
|\{(i,j)\in K_0:\, d_{ij}<\eps\}| \geq  |K_0| - (1-\delta) n^2+n
\geq (\delta  - \delta_0) n^2 + \delta_0n \geq (\delta - \delta_0) n^2.
$$
Therefore, using Markov's inequality,
$$
\Prob(\Event^c) \leq \frac{\E(|\{(i,j)\in K_0\,:\, d_{ij}<\eps\}|)}{n^2(\delta-\delta_0)} =
\frac{1}{n^2(\delta-\delta_0)}\sum_{(i,j)\in K_0}\Prob\{d_{ij}<\eps\}.
$$

Now, we condition on the event $\Event$. Fix a vector $x\in \Inc(\rho,\delta)$.
By the definition
the set
$$
 K_2=K_2(x):=\{(i,j)\in [n]\times [n]\, : \, |\bar{x}_i-\bar{x}_j|>\rho\}
$$
contains at least $\delta n^2$ elements. Clearly,
$K_2\subset K$. Thus we have $K_1\cup K_2\subset K$ and
$$
 |K_1|+|K_2|>\delta n^2 - n +n^2(1-\delta)=n(n-1)=|K|.
$$
 Hence $K_1\cap K_2\neq\emptyset$.
Choose $(i_0, j_0)\in K_1\cap K_2$. Then
$$
\Vert \bar x  A\Vert_2\geq |\bar{x}_{i_0}-\bar{x}_{j_0}|d_{i_0j_0}>\rho\,\varepsilon.
$$
Summarizing, we have shown that
\begin{align*}
\Prob\big\{\inf_{x\in \Inc(\rho,\delta)} \Vert
\bar x  A\Vert_2 \leq \varepsilon \rho\big\}\le\Prob(\Event^c)
\le \frac{1}{n^2(\delta-\delta_0)}\sum_{(i,j)\in K_0}
\Prob\{d_{ij}<\eps\}.
\end{align*}
\end{proof}

\medskip

The above lemma will be used to reduce the question of bounding the smallest singular value
to estimating distances between rows or columns of our random matrix and certain linear subspaces of $\C^n$.
In order to estimate the distance between the first row $\row_1$ and
$\spn\{\row_1+\row_2,\row_3, \row_4, \ldots,  \row_n\}$ of a random matrix, we will
need the following lemma. Its proof uses similar linear algebraic arguments as an earlier
work \cite{Cook-circ} (see Lemma~6.2 there). However Lemma~\ref{lem-distance-rows-deterministic}
significantly differs from  \cite[Lemma~6.2 there]{Cook-circ} and works for a general square
matrix. We apply it later with $v$ being the vector at which $s_n(A)$ attains (see the definition of
$f(A)$ below).

\begin{lemma}\label{lem-distance-rows-deterministic}
Let $A$ be an $n\times n$ complex matrix (either deterministic or random) and denote
$\row_i:=\row_i(A)$, $i\leq n$. Further, let $A^{1,2}$ be the $(n-2)\times n$ matrix
obtained by removing the first two rows of $A$, and let $Y\subset \C^n$ be the
linear span  of $\row_1+\row_2$, $\row_3$, $\row_4, \ldots$, $\row_n$.
Then for every unit complex vector $v\in\C^n$ we have
$$
\d\big(\row_1, Y\big)\geq
\frac{s_{n}(A)\,|\langle \bar \row_1,v\rangle|}{s_{n}(A)+\Vert  A^{1,2} v \Vert_2+
|\langle \bar \row_1+\bar \row_2 ,v\rangle|}.
$$
In particular, if a unit complex vector $v\in\C^n$ satisfies
$$
   \Vert  A^{1,2} v \Vert_2\leq s_{n}(A) \quad \mbox{ and } \quad
   |\langle \bar  \row_1 +\bar \row_2, v\rangle|\leq 2s_{n}(A)
$$
then
$$
 \d\big(\row_1, Y\big)\ge{\vert\langle \bar \row_1, v\rangle\vert}/4 .
$$
\end{lemma}

\begin{proof}
Let $x$
be a vector from $Y$, i.e.\  $x = b(\row_1+\row_2)+\sum_{i= 3}^n a_i \row_i$
for some $b,a_3,a_4,\dots,a_n\in\C$.
Fix a unit vector $v\in\C^n$.
We clearly have
\begin{equation}\label{eq: aux 2-958}
 \Vert \row_1-x\Vert_2\geq \vert\langle \row_1-x,\bar v \rangle\vert
\ge \vert\langle \row_1,\bar v \rangle\vert -\vert \langle x,\bar v \rangle \vert.
\end{equation}
Consider the vector
$y:=(1-b,-b,-a_3,\ldots, -a_n)$. Then, $\row_1-x=A^T y$, whence
$$
\Vert \row_1-x\Vert_2\geq s_{n}(A^T) \Vert y\Vert_2= s_{n}(A) \Vert y\Vert_2.
$$
Therefore, using the Cauchy--Schwarz inequality, we obtain
\begin{align*}
\vert \langle x,\bar v \rangle \vert&\leq |b|\vert \langle \row_1+\row_2,\bar v \rangle\vert
+\Big(\sum_{i= 3}^n |a_i|^2\Big)^{\frac{1}{2}}
\Big(\sum_{i= 3}^n |\langle \row_i,\bar v \rangle|^2\Big)^{\frac{1}{2}}\\
&\le \Vert y\Vert_2 \left( \vert \langle \row_1+\row_2,\bar v \rangle\vert+\Vert  A^{1,2} v \Vert_2\right)
\\
&\le \frac{1}{s_{n}(A)}\, \Vert \row_1
-x\Vert_2 \left( \vert \langle \row_1+\row_2,\bar v \rangle\vert+\Vert A^{1,2} v\Vert_2\right).
\end{align*}
This, together with \eqref{eq: aux 2-958}, implies that
$$
\Vert \row_1-x\Vert_2\geq \frac{s_{n}(A)\,\vert\langle \row_1,\bar v \rangle\vert}
{s_{n}(A)+\vert \langle \row_1+\row_2,\bar v \rangle\vert
+\Vert A^{1,2} v\Vert_2}.
$$
The lemma follows by taking the infimum over $x\in Y$.
\end{proof}

We would like to note that for a unit vector $v_0$, orthogonal to the span of
$\row_1+\row_2$, $\row_3$, $\row_4, \ldots$, $\row_n$,
we have $A^{1,2}\bar v _0=0$ and $\langle \row_1+\row_2,v_0\rangle=0$, so the lemma applied to $\bar v _0$
gives a trivial bound $\d\big(\row_1, Y\big)\ge{\vert\langle \row_1,v_0\rangle\vert}$.
Thus Lemma~\ref{lem-distance-rows-deterministic} can be viewed as a ``continuous'' version of this trivial estimate.

\subsection{Proof of Theorem~\ref{th-ssv}}

For the rest of the section, we fix a function $f$ on the set of $n\times n$ complex matrices,
which associates with every matrix $A$ a complex vector $f(A)$ such that
$\Vert Af(A)\Vert_2=s_{n}(A)$.
Note that in general the corresponding singular vector
is not uniquely defined, so we fix {\it some} vector $f(A)$ satisfying the above condition.
Since we work with shifted matrices, we also adopt another notation: given a (fixed) complex matrix $\WW$, by
$f_{\WW}$ we denote the function on the set of $n\times n$ matrices defined
by $f_{\WW}(A):=f(A+\WW)$.

\smallskip

Fix parameters $\kappa,\rho>0$, $\delta\in (1/\sqrt{d},1)$ and a complex matrix $\WW$
(note that for $\delta\leq 1/\sqrt{d}$ the bound for probability in Theorem~\ref{th-ssv} becomes
greater than one, hence the theorem holds automatically).
For the rest of the section, we assume that the parameters are given, and will specify each time
what restrictions on the numbers $\kappa,\rho,\delta,d$ and the matrix $\WW$ we impose.
Further, define
$$
  \eps_1=\eps_1(\delta):=\delta/ (C_1 \log (2e/\delta),
$$
where $C_1$ is a sufficiently large absolute constant (it is enough to take
 the constant from Proposition~\ref{p:zero minor} multiplied by $9$).
Set $\alpha:=\delta/(9\eps_1 d)$ and  $\beta:=\delta/2$.
Note that with such a choice of $\alpha$, $\beta$ we have
$\alpha \geq  (C  \log (e/\beta))/{d}$ and, using that $\delta >1/\sqrt{d}$ and that $d$ is  large enough,
we also have $\alpha \leq \min (\beta, 1/4)$. In other words the conditions of Proposition~\ref{p:zero minor}
are satisfied.
Let $\Om0=\Om0(\alpha, \beta)$ and $\O1$ be the events defined in and after Proposition~\ref{p:disjoint-rows}.
Define the event
$$
 \Event_{0}=\Event_0(\WW,\kappa,\rho,\delta):=\Om0^c\cap\O1\cap\Event_{\ref{s: quantitative}}.
$$
In words, $\Event_0$ corresponds to the set of matrices in $\Mc$ without large zero submatrices,
with almost no overlap between supports of any two rows or columns, and with the structural assumption
on vectors ``close'' to the kernel of the respective shifted matrix.
Note that under  assumptions of Theorem~\ref{th-ssv}, by Propositions~\ref{p:disjoint-rows} and \ref{p:zero minor}
we have
\begin{equation} \label{coroftwoprop}
\Prob(\Event_{0}^c)\le 2n^{-2}
\end{equation}
(the assumption on $d$ in Theorem~\ref{th-ssv} comes from $d\leq \eps _1 n / 6$ needed in Propositions~\ref{p:disjoint-rows}).

\medskip

The next lemma shows, roughly speaking, that there are relatively few matrices $M\in\Mc$
such that the corresponding singular vector $f_{\WW}(M)$ is ``almost constant''
when restricted to supports of a large number of rows of $M$. It is similar to
Lemmas~4.15 and 4.16 in \cite{LLTTY:15} and to Lemma~6.2 in \cite{Cook-circ}.

\begin{lemma} \label{l:bad rows}
Assume that $d$ is large enough.
For every pair of indices $\ell \ne i$ define the event
\begin{align*}
\Event_{\ref{l:bad rows}}^{\ell, i} : =&\big \{M\in \Event_{\ref{th-ssv}}\cap\Event_{0} \, :\, \,
\exists \lambda\in\C \mbox{ such that }
\\
&\big\vert \{j\in\supp(\row_\ell(M)+\row_i(M)):\,|(f_{\WW}(M))_j-\lambda|\leq \rho/4\}\big\vert
> 2(1-2\eps _1)d \big\}.
\end{align*}
Then for every (fixed) $\ell \leq n$ one has
\begin{align*}
\sum_{i:\,i\ne \ell} |\Event^{\ell, i}_{\ref{l:bad rows}}|  \le \frac{\delta n}{9\eps _1 d}  \, |\Mc|.
\end{align*}
\end{lemma}
\begin{proof}
Without loss of generality, we can assume that $\ell =1$. Let $\Event$ denote the event
\begin{align*}
\big\{M\in \Event_{\ref{th-ssv}}\cap\Event_{0}\, :\, \,
&\exists \lambda\in\C \mbox{ with }\\
&\big\vert \{j\in\supp \row_1(M):|(f_{\WW}(M))_j-\lambda|\leq \rho/2\}\big \vert
>(1-4\eps _1)d \big\}.
\end{align*}
Note that  $\Event_{\ref{l:bad rows}}^{1, i}\subset\Event$ for every $i\geq 2$. Indeed if $M\in \Event_{\ref{l:bad rows}}^{1, i}$ for some $i\geq 2$, then
there exists $\lambda\in \C$ such that
$$
\big\vert \{j\in\supp(\row_1(M)+\row_i(M)):\,|(f_{\WW}(M))_j-\lambda|\leq \rho/4\}\big\vert
> 2(1-2\eps _1)d.
$$
Therefore
\begin{align*}
\big\vert \{&j\in\supp \row_1(M): |(f_{\WW}(M))_j-\lambda|\leq \rho/2\}\big \vert
\geq \big\vert \{j\in\supp \row_1(M):|(f_{\WW}(M))_j-\lambda|\leq \rho/4\}\big \vert \\ &\\
&\geq \big\vert \{j\in\supp(\row_1(M)+\row_i(M)):\,|(f_{\WW}(M))_j-\lambda|\leq \rho/4\}\big\vert- \big\vert \supp(\row_i(M)\big\vert\\
&\geq (1-4\eps _1)d,
\end{align*}
which means that $M$ belongs to $\Event$.
\smallskip
For every $M\in \Event$,  fix a number $\lambda_0=\lambda_0(M)\in\C$
such that
\begin{equation}\label{eq: bad-rows}
\big\vert \big\{j\in\supp \row_1(M):\,|(f_{\WW} (M))_j-\lambda_0|\leq \rho/2\big\}\big\vert>(1-4\eps _1)d.
\end{equation}

Now, take any $M\in \Event$ and let
$$
  J_M:=\{j\le n:\, |(f_{\WW}(M))_j-\lambda_0|\leq \rho\}.
$$
Since $M\in \Event _0 \subset \Event_{\ref{s: quantitative}}$ (i.e.,  all vectors ``close'' to the kernel of $M+\WW$
are essentially non-constant) and $\|(M+\WW)f_{\WW}(M)\|_2\leq\kappa$, we have $|J_M|\le (1-\delta) n$. Let also
$$
I_{M}:=\{i\le n:\, |\supp \row_i(M)\cap J_M|\ge (1-4\eps _1)d \}.
$$
We first show that $|I_{M}|\leq  \delta n/(9\eps _1 d)$. Assume the opposite.
Choose a set  $\widetilde{I}\subset I_{M}$ with $|\widetilde{I}|=\left\lceil \delta n/(9\eps _1 d) \r\rceil$.
Clearly,
$$
 \forall (i,j)\in \widetilde{I}\times \big(\cup_{i\in \widetilde{I}}\supp \row_i(M)\big)^c \quad  \mbox{ one has } \quad \mu_{ij}=0,
$$
and
\begin{align*}
\big|\big(\cup_{i\in \widetilde{I}}\supp \row_i(M)\big)^c\big|&\ge n-|J_M|-
\big|\cup_{i\in \widetilde{I}}\supp \row_i(M)\setminus J_M\big|\ge \delta n -4\eps _1 d |\, \widetilde{I}|\ge \delta n/2.
\end{align*}
This contradicts the assumption $M\in\Om0^c$ (no large zero-submatrices).

\smallskip

By the definition of $I_M$, for every $i\in I_{M}^c$,
$$
 |\{j\in\supp \row_i(M)\, :\, |(f_{\WW}(M))_j-\lambda_0|>\rho\}|\ge 4\eps _1 d.
$$
This implies for every $i\in I_{M}^c$ and for every $\lam$ satisfying $|\lambda-\lambda_0|\le 3\rho/4$,
$$
|\{j\in\supp (\row_1(M)+\row_i(M))\, :\,|(f_{\WW}(M))_j-\lambda|> \rho/4\}|\ge 4\eps _1 d.
$$
Using the triangle inequality together with \eqref{eq: bad-rows}, we also observe that for every
$\lam$ satisfying $|\lambda-\lambda_0| > 3\rho/4$,
\begin{align*}
|\{j\in\supp (\row_1(M)+\row_i(M))&\,:\,|(f_{\WW}(M))_j-\lambda| > \rho/4\}|\\
&\geq |\{j\in\supp (\row_1(M)+\row_i(M))\,:\,|(f_{\WW}(M))_j-\lambda_0| \leq  \rho/2\}|\\
&\geq |\{j\in\supp \row_1(M)\,:\,|(f_{\WW}(M))_j-\lambda_0| \leq \rho/2\}|\\
&\ge (1-4\eps _1) d \geq 4\eps _1 d.
\end{align*}
Thus for every $i\in I_{M}^c$ and every $\lambda\in \C$ we obtain
$$
 \big|\big\{j\in\supp (\row_1(M)+\row_i(M))\,:\,|(f_{\WW}(M))_j-\lambda|\leq  \rho/4\big\}\big|
 \le 2d  -4\eps _1 d.
$$
This proves that
 for every $M\in\Event$ and  $i\in I^c_{M}$ one has $ M\in\Event\setminus\Event^{1, i}_{\ref{l:bad rows}}$. Therefore,
\begin{align*}
\sum_{i=2}^n|\Event^{1, i}_{\ref{l:bad rows}}|=
\sum_{i=2}^n\sum_{M\in\Event}\chi_{\{M\in\Event^{1, i}_{\ref{l:bad rows}}\}}=
\sum_{M\in\Event}\sum_{i=2}^n\chi_{\{M\in\Event^{1, i}_{\ref{l:bad rows}}\}}
\le
\sum_{M\in\Event}|I_{M}|\le \frac{\delta n\,  |\Event|}{9\eps _1 d}    .
\end{align*}
\end{proof}

\begin{rem}
\label{r:bad rows}
Note that by Proposition~\ref{p:disjoint-rows} for every $i\neq \ell$ and every matrix $M\in\Event_{0}$
one has $|\supp R_i(M)\,  \cap\, \supp \row_\ell(M)|\leq 2\eps_1 d$. Therefore, for every $i\neq \ell$
and every matrix  $M\in( \Event_{\ref{th-ssv}}\cap\Event_{0})\setminus
\Event_{\ref{l:bad rows}}^{\ell, i}$, one has
\begin{align*}
 \forall\lambda\in\C \quad &\big\vert \{j\in \supp R_\ell(M)\,  \triangle\, \supp \row_i(M)\, :\,|(f_{\WW}(M))_j-\lambda|> \rho/4\}\big\vert\\
&>|\supp \row_i(M)\,  \triangle\, \supp \row_\ell(M)| -2d + 4\eps _1 d \geq 2 \eps _1 d ,
\end{align*}
where $\triangle$ denotes the symmetric difference of sets.
\end{rem}

\smallskip

The next observation is a direct consequence of Lemma~\ref{l:RV} and Lemma~\ref{l:bad rows}.

\begin{cor}\label{cor-pairs}
Assume that $0< \delta<1$,
and that $d$ satisfies the assumptions of Lemma~\ref{l:bad rows}.
Then there exists a pair $(\ell, j)\in [n]\times [n]$ with $\ell \ne j$ such that
\begin{equation}\label{cor-pairs-1}
|\Event^{\ell, j}_{\ref{l:bad rows}}|  \le \frac{1}{4\eps _1 d}\, |\Mc|.
\end{equation}
Moreover, setting $\row^\WW_i=\row^\WW_i(M):=\row_i(M+\WW)$ for  all $i\leq n$ and $M\in\Mc$, we have
for any $\varepsilon>0$,
\begin{align*}
 \big\vert\big\{M &\in \Event_{\ref{th-ssv}}\cap\Event_{0}\, :\, \inf_{x\in \Inc(\rho,\delta)} \Vert
 \bar x (M+\WW)\Vert_2 \leq \varepsilon \rho\big\}\big\vert
\\
 &\leq
 2 \big\vert\big\{M \in \Event_{\ref{th-ssv}}\cap\Event_{0}\, :\,
  \d\big(\row_\ell^\WW, \spn\big\{\{\row_k^\WW\}_{k\ne \ell, j},\,\row_\ell^\WW+\row_j^\WW\big\}\big)
< \varepsilon \big\}\big\vert/\delta.
\end{align*}
\end{cor}
\begin{proof} Denote $K:=\{(\ell,j)\, :\, 1\leq \ell \ne j\leq n\}$. Set $\delta _0=\delta/2$.
 Lemma~\ref{l:bad rows} implies that for every fixed $\ell \leq n$ there are at least $(1-\delta _0) (n-1)$ choices of
 $j\ne \ell$  satisfying  \eqref{cor-pairs-1}. Therefore, the subset
$$
  K_0:=\big\{(\ell, j)\in K \, :\, (\ell, j) \, \, \mbox{ satisfies }\, \,   \eqref{cor-pairs-1}\big\}
$$
 has cardinality at least $(1-\delta _0) n(n-1)$.  Choosing a pair $(\ell, j)\in K_0$ with maximal
$$
 \big\vert\big\{M \in \Event_{\ref{th-ssv}}\cap\Event_{0}\, :\,
  \d\big(\row_\ell^\WW, \spn\big\{\{\row_k^\WW\}_{k\ne \ell, j},\,\row_\ell^\WW+\row_j^\WW\big\}\big)< \varepsilon \big\}\big\vert
$$
and applying Lemma~\ref{l:RV} to the random matrix $A=M+\WW$, where $M$ is uniformly distributed in
$\Event_{\ref{th-ssv}}\cap\Event_{0}$, we obtain the desired result.
\end{proof}

Corollary~\ref{cor-pairs} reduces the question of bounding the infimum over ``non-constant'' vectors
to calculating the distance between a particular matrix row and corresponding linear span,
and additionally makes sure that the singular vector $f_{\WW}(M)$ is essentially non-constant
when restricted to the union of the supports of $j$-th and $\ell$-th rows. The latter allows to apply
Littlewood--Offord--type anti-concentration statements.
Note that, instead of bounding the cardinality of the event $\Event_{\ref{th-ssv}}$ directly,
we will bound the cardinality of the intersection of $\Event_{\ref{th-ssv}}$ with a ``good'' event $\Event_0$,
and then use the fact that $\Event_0^c$ is small (under the assumptions of the theorem).

\bigskip

We are now ready to describe a partition of the event $\Om0^c\cap\O1$,
which will be used in the proof of Theorem~\ref{th-ssv}.
Fix $d$, parameters $\rho,\delta$ and complex matrix $\WW$.
Let $\kappa$ be defined as in Theorem~\ref{th-ssv} and assume that
all the conditions of the theorem (including assumptions on the parameters) are satisfied.
Let the pair $(\ell, j)$ be given by Corollary~\ref{cor-pairs}.
From now on, to simplify notation, we will assume that $(\ell, j)=(1,2)$.
We would like to  emphasize that the proof below can be carried for any admissible pair $(\ell,j)$
by simply adjusting indices.

Consider a set of $(n-2)\times n$ matrices
$$
\HH:=\{M^{1,2}:\, M\in\Om0^c\cap\O1\}.
$$
For every $H\in\HH$, let $C_H$ be the equivalence class of matrices sharing the same $(n-2)\times n$ submatrix, that is
$$
C_H:=\{M\in\Om0^c\cap\O1:\, M^{1,2}=H\}.
$$
Note that for  $M_1, M_2\in C_H$ one has $\row_1(M_1)+\row_2(M_1)=\row_1(M_2)+\row_2(M_2)$, that is
the intersection and the union of the supports of the first two rows is the same for all matrices in the class:
$$
  S_1=S_1(H):=\supp \row_1(M_1)\cap \supp \row_2(M_1)= \supp \row_1(M_2)\cap \supp \row_2(M_2)
$$
and
$$
 S_2=S_2(H):=\supp \row_1(M_1)\cup \supp \row_2(M_1)= \supp \row_1(M_2)\cup \supp \row_2(M_2).
$$
In particular, $|C_H|={2m \choose m}$, where $m=m(H)=|S_2\setminus S_1|$
is the cardinality of the symmetric difference of the supports of the first two rows for any matrix in $C_H$.
Observe that, because our matrices belong to $\O1$, we have
$m(H)\geq 2(1-\eps_1)d$.
In every class $C_H$, fix a subset $\Cl\subset C_H$ of matrices satisfying
$$
\forall \widetilde{M}\in \Cl\,\,\forall {M}\in C_H\setminus \Cl: \;\; s_{n}(\widetilde{M}+\WW)\le s_{n}(M+\WW)
\;\; \mbox{and}\;\;
\frac{1}{2\sqrt{\eps _1 d}}  \leq \frac{|\Cl|}{|C_H|} \leq\frac{1}{\sqrt{\eps _1 d} }.
$$
Thus, $\Cl$ is the set of matrices $\widetilde M$ delivering a ``small'' minimal singular value of $\widetilde M+\WW$,
compared to other matrices in $C_H$.
Denote $\Event_{\ref{l:bad rows}}:=\Event_{\ref{l:bad rows}}^{1,2}$ and define
\begin{align*}
\HH_1:=\{H\in\HH:\,\Cl\cap\Event^c_{\ref{th-ssv}}\neq\emptyset\}, \,\,\,
\HH_2:=\{H\in\HH_1^c:\,\Cl\subset\Event^c_{\ref{s: quantitative}}\cup\Event_{\ref{l:bad rows}}\},
\,\,\, \HH_3:=\HH_1^c\setminus \HH_2.
\end{align*}
Roughly speaking, the set $\HH_1$ is the collection of all $(n-2)\times n$ submatrix such that
a vast majority of the corresponding shifted matrices have ``large'' smallest singular value.
The set $\HH_2$ is the set of all submatrices not in $\HH_1$ such that the corresponding shifted matrices
have ``bad'' characteristics in regard to their ``almost null'' vectors as well as the vectors delivering the smallest singular value.
Finally, $\HH_3$ is all the remaining submatrices. It is the third category which is the most interesting for us and
which will require Littlewood--Offord type anti-concentration arguments.

Consider the partition
\begin{align}
\label{partition}
\Om0^c\cap\O1=\bigcup_{H\in\HH_1}C_H \cup\bigcup_{H\in\HH_2}C_H \cup \bigcup_{H\in\HH_3}C_H.
\end{align}
We will analyze separately each of the sets $\bigcup_{H\in\HH_i}C_H$, $i\leq 3$. First we show that
for $i=1,2$ the respective unions have a small cardinality.

By the definition of $\Event_{\ref{th-ssv}}$, for every $H\in\HH_1$
there exists a matrix $M\in\Cl$ with $s_{n}(M+\WW)>\kappa$. Hence, by the definition of $\Cl$,
$$
|\{M\in C_H:\, s_{n}(M+\WW)\le\kappa\}|\le |\Cl| \leq \frac{1}{\sqrt{\eps _1 d} }\, |C_H|,
$$
which implies
\begin{equation}
\label{H1c}
\Big\vert \bigcup_{H\in\HH_1}C_H\cap\Event_{\ref{th-ssv}}\Big\vert\le \frac{1}{\sqrt{\eps _1 d} }\, |\Mc|.
\end{equation}
Further, by the definitions of $\Cl$ and $\HH_2$, the assumptions of Theorem~\ref{th-ssv}, and
Corollary~\ref{cor-pairs}, we have
\begin{equation}
\label{H2}
\Big\vert \bigcup_{H\in\HH_2}C_H\Big\vert\leq 2\sqrt{\eps _1 d} \, \sum_{H\in\HH_2}\vert\Cl\vert
\le 2\sqrt{\eps _1 d} \, | \Event^c_{\ref{s: quantitative}}
\cup\Event_{\ref{l:bad rows}}|\le 2\sqrt{\eps _1 d} \,\Big(n^{-2}+\frac{1}{4\eps _1 d}\Big) |\Mc|.
\end{equation}

Regarding the set  $\HH_3$, we prove the following lemma.

\begin{lemma}\label{l:H3}
Denoting $\row_i^{\WW}=\row_i^{\WW}(M) :=\row_i(M+\WW)$, for $i\leq n$ and $M\in\Mc$, we have
\begin{equation*}
\Big\vert \Big\{M\in\bigcup_{H\in\HH_3}C_H:\, \d\big(\row_1^{\WW}, \spn\big\{\{\row_k^{\WW}\}_{k> 2},\,
\row_1^{\WW}+\row_2^{\WW}\big\}\big)
< \rho/16\Big\} \Big\vert \le C (\eps _1 d)^{-1/2}\, |\Mc|,
\end{equation*}
where $C>0$ is a universal constant.
\end{lemma}
\begin{proof}
The set $\HH_3$ can be equivalently written as
$$
  \{H\in\HH\, :\, \Cl\subset  \Event_{\ref{th-ssv}} \quad \mbox{and} \quad \Cl
\cap \Event_{\ref{s: quantitative}}\cap\Event_{\ref{l:bad rows}}^c \neq \emptyset\} .
$$
Fix any $H\in\HH_3$ and a matrix $\widetilde{M}\in\Cl\cap\Event_{\ref{s: quantitative}}\cap\Event_{\ref{l:bad rows}}^c$.
For every $M\in C_H\setminus\Cl$ we have
$$
 \Vert (M+\WW)^{1,2}f_{\WW}(\widetilde{M})\Vert_2
 =\Vert (\widetilde{M}+\WW)^{1,2}f_{\WW}(\widetilde{M})\Vert_2
 \le s_{n}(\widetilde{M}+\WW)\le s_{n}(M+\WW)
$$
and
\begin{align*}
|\langle (\bar R_1(\widetilde{M}+\WW)) &+(\bar R_2(\widetilde{M}+\WW)),\,f_{\WW}(\widetilde{M})\rangle|
\\ &
\le 2 \Vert (\widetilde{M}+\WW) f_{\WW}(\widetilde{M})\Vert_2 =
 2 s_{n}(\widetilde{M}+\WW)\le 2s_{n}(M+\WW).
\end{align*}
This and Lemma \ref{lem-distance-rows-deterministic} applied to the matrix $M+\WW$
imply that for at least
$$
  |C_H|-|\Cl| \geq \big(1-1/\sqrt{\eps _1 d}\big)|C_H|
$$
 matrices $M\in C_H$, one has
\begin{align*}
 \d\big(R_1^{\WW}(M), \spn\big\{\{R_k^{\WW}(M)\}_{k> 2},\,R_1^{\WW}(M)+R_2^{\WW}(M)\big\}\big)
 \ge|\langle (\bar R_1(M+\WW)),\,f_{\WW}(\widetilde{M})\rangle|/4.
\end{align*}
The following claim, whose proof we postpone,  completes the proof of the lemma.
\end{proof}
\begin{claim}\label{claim-KT}
With the above notation, for every $H\in\HH_3$ and
$\widetilde{M}\in\Cl\cap\Event_{\ref{s: quantitative}}\cap\Event_{\ref{l:bad rows}}^c$ we have
\begin{align*}
\vert \{M\in C_H:\, |\langle (\bar R_1({M+\WW})),\,f_{\WW}(\widetilde{M})\rangle|< \rho/4\}\vert\le c (\eps _1 d)^{-1/2}\, |C_H|
\end{align*}
for some universal constant $c>0$.
\end{claim}

\begin{proof}[Proof of Theorem~\ref{th-ssv}]
Recall that $\Event_{0}=\Om0^c\cap\O1\cap\Event_{\ref{s: quantitative}}$ and that $\kappa=\rho^2/16$.
By (\ref{coroftwoprop}) we have
$$
 |\Event_{\ref{th-ssv}}|\le |\Event_{\ref{th-ssv}}\cap\Event_{0}|+|\Event^c_{0}|
 \leq  |\Event_{\ref{th-ssv}}\cap\Event_{0}|+ 2n^{-2}|\Mc|.
$$
Next, using the definitions of the events $\Event_{\ref{th-ssv}}$, $\Event_{\ref{s: quantitative}}$, and $\Event_{0}$, we observe that
\begin{align*}
|\Event_{\ref{th-ssv}}\cap\Event_{0}|&=\big|\{M\in\Event_{\ref{th-ssv}}\cap\Event_{0}:\,
 \inf_{\|x\|_2=1} \Vert \bar x (M+\WW)\Vert_2 \le\rho^2/16\}\big|
\\
&=\big|\{M\in\Event_{\ref{th-ssv}}\cap\Event_{0}:\, \inf_{x\in \Inc(\rho,\delta)} \Vert
\bar x (M+\WW)\Vert_2 \le\rho^2/16\}\big|.
\end{align*}
Recall that we agreed to assume that the pair of indices $(1,2)$ satisfies the conditions in Corollary~\ref{cor-pairs}.
In particular, this implies for $\row_i^{\WW}:=\row_i(M+\WW)$, $i\leq n$,
\begin{align*}
  |\Event_{\ref{th-ssv}}\cap\Event_{0}|  \leq
  (2//\delta) \, \big\vert\big\{M\in \Event_{\ref{th-ssv}}\cap\Event_{0}:\,
  \d(\row_1^{\WW}, \spn\big\{\{\row_k^{\WW}\}_{k\geq 2},\,\row_1^{\WW}+\row_2^{\WW}\big\})< \rho/16 \big\}\big\vert.
\end{align*}
Finally estimates  \eqref{partition}--\eqref{H2} and Lemma~\ref{l:H3} imply that
$$
 |\Event_{\ref{th-ssv}}\cap\Event_{0}|
\leq \frac{C'\, |\Mc|}{\delta \sqrt{\eps _1 d}}
$$
for a universal constant $C'>0$.
Since  $\eps _1=  \delta/ (C_1 \log (e/\delta))$, this implies the desired result.
\end{proof}

\subsection{Proof of Claim~\ref{claim-KT}}

We will use the notations from Lemma~\ref{l:H3} of the previous subsection.
Recall that
$$\widetilde{M}\in\Cl\cap\Event_{\ref{s: quantitative}}\cap \Event_{\ref{l:bad rows}}^c\subset
\Event_{0} \cap \Event_{\ref{th-ssv}}\cap \Event_{\ref{l:bad rows}}^c
$$
and that
$$
 S_1(H) = \supp \row_1(M)\cap \supp \row_2(M), \quad
 S_2(H) = \supp \row_1(M)\cup \supp \row_2(M)
$$
do not depend on the choice
of $M\in C_H$.
Denote
$$S_3:=S_2\setminus S_1 = \supp \row_1(M)\,\triangle\, \supp \row_2(M).$$
Take $y:= f_{\WW}( \widetilde M)$.
Using Remark~\ref{r:bad rows} and applying Lemma~\ref{l: quantified claim4.7} to the vector
$\{y_j\}_{j\in S_3}$ we find two disjoint sets $A_1, A_2\subset S_3$ with cardinalities
$|A_1|,|A_2|\geq \ell:=\lceil \eps_1 d/2\rceil$  and such that  for all $i\in A_1$ and $j\in A_2$ one has
$|y_i-y_j|\geq \rho/(4\sqrt{2})$.
For the rest of the proof, we fix $\ell$ couples of distinct indices
$(i_1,j_1),(i_2,j_2),\dots,(i_{\ell},j_{\ell})\in A_1\times A_2$.
Next, we define auxiliary subsets of $C_H$ as follows: for any subset $I\subset [\ell]$
and any $S\subset S_3\setminus \bigcup_{k\in I}\{i_k,j_k\}$
we set
\begin{align*}
{\rm cpl}(I,S):=\Big\{&M\in C_H\,:\, \{k\, : \, |\supp \row_1(M)\cap \{i_k,j_k\}|=1\} = I \, \, \, \mbox{ and}\\
& \supp \row_1(M)\setminus \Big(S_1\cup\bigcup_{k\in I}\{i_k,j_k\}\Big)
=S\Big\}.
\end{align*}
Roughly speaking, each subclass ${\rm cpl}(I,S)$ is obtained by picking a subset of the couples $(i_k,j_k)$ on which
the first row of a matrix is ``allowed to vary'' while fixing all other coordinates of $R_1$.
Note that subclasses ${\rm cpl}(I,S)$ can be empty for some $I,S$
and that the collection $\{{\rm cpl}(I,S)\}_{I,S}$ (taking all admissible $I$, $S$) forms a partition of the class $C_H$. Observe that
\begin{equation}\label{eq: aux 435234}
\Big|\bigcup_{|I|\leq \ell/4,\,S}{\rm cpl}(I,S)\Big|\leq  \frac{1}{d^2}\, |C_H|,
\end{equation}
where the union is taken over all subsets $I\subset[\ell]$
of cardinality at most $\ell /4$ and all admissible sets $S$, and  where $d$ is large enough.
Indeed, recall that the class $C_H$ can be identified
via a natural bijection with the collection of all $m$-element subsets of $[2m]$,
where $m:=|S_3|/2$.
With such an identification and by choosing an appropriate permutation of $[2m]$,
the set of matrices on the left hand side of \eqref{eq: aux 435234}
corresponds to the collection of $m$-element subsets $B$ of $[2m]$ such that $|\{k\leq \ell\, :\,
|B\cap\{k,k+\ell\}|=1\}|\leq \ell/4$, where
$$
 \ell\geq \eps_1 d/2 = \delta d/(2C_1 \log (2e/\delta) )\geq \sqrt{d}/(e \log (30 d)).
$$
 Then a direct calculation shows that for large $d$ the number of
such subsets $B$ is much less than $(\eps_1 d)^{-2}\, {2m \choose m}$.

\smallskip

As the final step in the proof of the claim, we fix a non-empty subclass ${\rm cpl}(I,S)$ with $|I|> \ell/4$
and observe that $|{\rm cpl}(I,S)|=2^{|I|}$.
In fact, each matrix $M$ in ${\rm cpl}(I,S)$ can be uniquely determined by picking either $i_k$ or $j_k$
for every $k\in I$ and then defining the support of the first row of $M$ as the union of the chosen indices,
the set $S$ and the intersection part $S_1$.
Moreover, for each $M\in {\rm cpl}(I,S)$ the inner product $\langle (\bar \row_1({M+\WW})),\, y\rangle$
 can be written as
$$
  \langle (\bar  \row_1({M+\WW})),\,y \rangle = \langle \bar \row_1({M}),\,y \rangle
  + \langle \bar  \row_1({\WW}),\,y \rangle =
  U+\sum_{k\in I}\xi_k(M) (\bar{y}_{i_k}-\bar{y}_{j_k}),
$$
where $U$ is a complex number which is {\it the same} for all $M\in {\rm cpl}(I,S)$, and
$\xi_k(M)$, $k\in I$, are $0/1$-valued functions of $M$ defined as $\xi_k(M):=|\supp \row_1(M)\cap\{i_k\}|$.
In other words, $\xi_k(M)$ is the indicator of the event that the support of the first row of $M$ contains $i_k$ and  not $j_k$.
It is not difficult to see that the functions $\xi_k(M)$, $k\in I$, considered as random variables uniformly distributed
on ${\rm cpl}(I,S)$, are jointly independent; and that
 for each $k\in I$ one has
$$
 |\{M\in {\rm cpl}(I,S):\,\xi_k(M)=1\}|= |{\rm cpl}(I,S)|/2=2^{|I|-1}.
$$
Further, by our choice of the pairs $(i_k,j_k)$, we have $|\bar{y}_{i_k}-\bar{y}_{j_k}|=|y_{i_k}-y_{j_k}|\geq\rho/(4\sqrt{2})$
for all $k\in I$.
Note that $\eta _k = 2\xi _k (M)-1$, $k\in I$, are independent $\pm 1$ Bernoulli random variables and
that for every $v\in \C^I$,
$$
   \sum_{k\in I}\xi_k(M) v_k =  \sum_{k\in I}\eta_k(M) v_k/2 + \sum_{k\in I} v_k/2.
$$
Therefore, applying Proposition~\ref{p:Erdos}, we obtain
$$
 \big|\big\{M\in{\rm cpl}(I,\,S):\,|\langle (\bar  \row_1({M+\WW})),\,y \rangle| <  \rho/4\big\}\big|
\leq {c\, |{\rm cpl}(I,S)|}/{\sqrt{|I|}}
$$
for some universal constant $c>0$.
Taking the union over all $|I|>\ell /4$, we get
$$
   \Big|\Big\{M\in \bigcup_{|I|>\ell/4,\,S}{\rm cpl}(I,S):\,|\langle (\bar \row_1({M+\WW})),\,y \rangle|  <
   \rho/4\Big\}\Big|\leq 2c\big|C_H\big|/\sqrt{\ell}.
$$
Together with \eqref{eq: aux 435234}, this proves the claim.



\subsection{Proof of the main theorem}
\label{prmainth}

Here we explain how Theorems~\ref{lem: a-c}, \ref{t:steep}, and  \ref{th-ssv}
imply our main result, Theorem~\ref{mainth}. Fix $\rho =1/(d^{3/2} \bb)$, $\kappa =\rho^2/16$,
and $\delta =\nn/n\geq \aaa/\log d$. Then the condition on $d$ means
$d\leq c n/((\log d)(\log \log d))$.
Fix $\ww\in \C$ with $|\ww|\leq d/6$ and $\WW=-\ww\idmat$.
Recall that
$$
   \Inc(\rho,\delta) = \left(\C^n\setminus \BB (\rho)\r) \cap \left\{x\in\C^n:\,\|x\|_2=1\r\} .
$$
As it was mentioned in Remark~\ref{cortwoth}, Theorems~\ref{lem: a-c} and \ref{t:steep} (applied  twice for matrices and for
their conjugates) imply    that $\p (\Event_{\ref{s: quantitative}})\geq 1-1/n^2$. Thus, applying Theorem~\ref{th-ssv}, we
obtain
$$
  \Prob(\Event_{\ref{th-ssv}})\leq \frac{C_1\log^{3/2} d \,  \sqrt{\log \log d}}{\sqrt{d}},
$$
which implies the probability bound. Next,
$$
 \kappa = \rho^2/16=1/(16 d^3 \bb^2),
$$
where
$\bb = 4 d^{3/2} \hh_{r+1}$
in the case $n_0>1$ and $\bb=d\sqrt{n}$ if $n_0=1$.
This implies
$$
  s_{n} \geq
    \left\{
	\begin{array}{ll}
		 c/( p d^{6} n_1^{4+2\alpha _d})
       & \mbox{ if } \, \, \, n_0> p, \\
             (c \log d)/( d^{9})  & \mbox{ if }  \, \, \, 1<n_0\leq p,
\\
         c/(d^{5}n)  & \mbox{ if } \, \, \,  n_0=1,
	\end{array}
 \right.
$$
If $1<n_0\leq p$, then $d^2 \leq \a n \log d$ and $d^{2.5}\geq \a n \log^{1.5} d$, therefore
$$
 d^9/\log d\leq C_1 n^{4.5} \log ^{3.5} n.
$$
 If $n_0> p$, then, using the definition of
$\alpha _d$, we observe $s_{n} \geq d^{3/2}\log^{4.5} d/C_2 n^{4+2\alpha _d}$.
This implies the estimate in Theorem~\ref{mainth}.
\qed

\begin{rem}\label{sharpb}
In fact we proved that there exists absolute positive constants $c$, $C_1$, and $C_2$ such that
$$
  s_{n} \geq
    \left\{
	\begin{array}{ll}
		  c d^{3/2} \log ^{4.5} d\,\, n^{-4-2\alpha _d}
       & \mbox{\rm if }\, \, \,  d< c_1 n^{2/5} \log^{3/5} n, \\
		 c n^{-4.5}\log^{-3.5} n  & \mbox{\rm if } \, \, \,   c_1 n^{2/5} \log^{3/5} n\leq d < c_2 \sqrt{n \log n} ,
\\
         c/(d^{5} n)  & \mbox{\rm if } \, \, \,   c_2 \sqrt{n \log n}\leq  d\leq \frac{c n}{(\log n)(\log \log n)}
	\end{array}
 \right.
$$
with probability at least $(C_1\log^{3/2} d \,  \sqrt{\log \log d})/\sqrt{d}$.
\end{rem}





\address


\begin{thebibliography}{99}


\bibitem{ALPT}
{
R. Adamczak, A.E. Litvak, A. Pajor,
N. Tomczak-Jaegermann, {\it Quantitative estimates of the
convergence of the empirical covariance matrix in log-concave
Ensembles,} J. Amer. Math. Soc.  {\bf 23} (2010),  535--561.
}

\bibitem{AGLPT}
{
R. Adamczak, O. Guedon, A.E. Litvak, A. Pajor, N. Tomczak-Jaegermann,
{\it Condition number of a square matrix with i.i.d.\ columns drawn from a convex body,}
Proc. Amer. Math. Soc., {\bf 140} (2012), 987--998.
}


\bibitem{BS2010}
{
Z. Bai\ and\ J. W. Silverstein, {\it Spectral analysis of large dimensional random matrices}, second edition,
Springer Series in Statistics, Springer, New York, 2010. MR2567175
}

\bibitem{BCZ}
{
A. Basak, N. Cook, O. Zeitouni,
{\it
Circular law for the sum of random permutation matrices},
Electron. J. Probab. {\bf 23} (2018), no. 33, 1--51
}

\bibitem{BasRud}
{
A. Basak, M. Rudelson,
{\it Invertibility of sparse non-hermitian matrices},
 Adv. Math., {\bf 310} (2017),  426--483.
}

\bibitem{NLA1}
{
 D. Bau, L. Trefethen,
{\it  Numerical linear algebra}. Society for Industrial and Applied Mathematics
(SIAM), Philadelphia, PA, 1997.
}

\bibitem{BC}
{
C. Bordenave and D. Chafa\"{\i},
{\it Around the circular law}, Probab. Surv., {\bf 9} (2012), 1--89.
}

\bibitem{AGA2}
{
S. Brazitikos, A. Giannopoulos, P. Valettas, B. Vritsiou,
{\it Geometry of isotropic convex bodies.} Mathematical Surveys and Monographs, 196.
American Mathematical Society, Providence, RI, 2014.
}

\bibitem{Cook-digraphs}
{
N. A. Cook, Discrepancy properties for random regular digraphs,
Random Structures Algorithms {\bf 50} (2017), no.~1, 23--58. MR3583025
}

\bibitem{Cook-adjacency}
{
N. A. Cook, On the singularity of adjacency matrices for random regular digraphs,
Probab. Theory Related Fields {\bf 167} (2017), no.~1-2, 143--200. MR3602844
}

\bibitem{Cook-circ}
{
N. Cook, {\it The circular law for random regular digraphs}, ArXiv:1703.05839.
}



\bibitem{CV}
{
K.P.  Costello  and  V.  Vu, The rank of random graphs,  {\it Random Structures Algorithms} 33 (2008), no. 3, 269--285.
}


\bibitem{DavS}
K.R. Davidson, S.J. Szarek, {\it Local operator theory, random matrices and Banach spaces}.
Handbook of the geometry of Banach spaces, Vol. I, 317–366, North-Holland, Amsterdam, 2001.



\bibitem{Er}
{
P. Erd\H{o}s,
{\it On a lemma of Littlewood and Offord},
Bull. Amer. Math. Soc. {\bf 51} (1945), 898--902.
}

\bibitem{EY}
{
L. Erd\"os, H.-T. Yau, {\it A Dynamical Approach to Random Matrix Theory},
Courant Lecture Notes in Mathematics, {\bf 28}. Courant Institute of Mathematical Sciences,
New York; American Mathematical Society, Providence, RI, 2017.
}

\bibitem{Frieze ICM}
{
A.~Frieze, {\em  Random structures and algorithms},
Proceedings ICM, Vol. 1, 2014, 311--340.
}

\bibitem{GeMu}
{
S. van de Geer, A. Muro,
{\it
On higher order isotropy conditions and lower bounds for sparse quadratic forms},
Electron. J. Statist., {\bf 8} (2014), 3031--3061.
}

\bibitem{GLPT}
{
O. Guedon, A. E. Litvak, A. Pajor, N. Tomczak-Jaegermann,
On the interval of fluctuation of the singular values of random matrices,
J. Eur. Math. Soc. (JEMS) {\bf 19} (2017), no.~5, 1469--1505. MR3635358
}

\bibitem{Hoef}
{
W. Hoeffding, {\it
Probability Inequalities for Sums of Bounded Random Variables}
Journal of the American Statistical Association
Vol. 58, No. 301 (1963), 13--30
}

\bibitem{kashin}
{
  B. S. Kashin, {\it Diameters of some finite-dimensional sets and classes of
  smooth functions},  Izv. Akad. Nauk SSSR, Ser. Mat. {\bf 41}
  (1977), 334--351.
}

\bibitem{Kl}
{
D.J.~Kleitman,
{\it  On a Lemma of Littlewood and Offord on the Distributions
of Linear Combinations of Vectors}, Adv. Math., {\bf 5} (1970), 155--157.
}

\bibitem{KM}
{
V. Koltchinskii,  S. Mendelson,
{\it Bounding the smallest singular value of a random matrix without concentration},
Int. Math. Res. Notices. {\bf 23} (2015), 12991--13008.
}


\bibitem{LedT}
M. Ledoux, M. Talagrand,
{\it  Probability in Banach spaces. Isoperimetry and processes},
Springer-Verlag, Berlin, 1991.

\bibitem{LLTTY:15}
{
A.E.~Litvak, A.~Lytova, K.~Tikhomirov, N.~Tomczak-Jaegermann, and P.~Youssef,
{\it  Adjacency matrices of random digraphs: singularity and anti-concentration}, J. Math. Anal. Appl.,
  {\bf 445} (2017), 1447--1491.
}

\bibitem{LLTTY-cras}
{
A.E.~Litvak, A.~Lytova, K.~Tikhomirov, N.~Tomczak-Jaegermann, and P.~Youssef,
{\it  Anti-concentration property for random digraphs and invertibility of their adjacency matrices},
C.R. Math. Acad. Sci. Paris, {\bf 354} (2016), 121--124.
}



\bibitem{LLTTY-c}
{
A.E.~Litvak, A.~Lytova, K.~Tikhomirov, N.~Tomczak-Jaegermann, and P.~Youssef,
{\it  Circular law for sparse random regular digraphs}, submitted.
}


\bibitem{LLTTY-str}
{
A.E.~Litvak, A.~Lytova, K.~Tikhomirov, N.~Tomczak-Jaegermann, and P.~Youssef,
{\it  Structure of eigenvectors of random regular digraphs}, submitted.
}


\bibitem{LLTTY-rank}
{
A.E.~Litvak, A.~Lytova, K.~Tikhomirov, N.~Tomczak-Jaegermann, and P.~Youssef,
{\it  The rank of random regular digraphs of constant degree}, J. Complexity, to appear,
https://doi.org/10.1016/j.jco.2018.05.004.
}


\bibitem{LPRT}
{
A.E.~Litvak, A.~Pajor, M.~Rudelson, and N.~Tomczak-Jaegermann,
{\it Smallest singular value of random matrices and geometry of random polytopes,} Adv. Math.,
{\bf 195} (2005), 491--523.
}

\bibitem{LPRTV}
{
 A. E. Litvak, A. Pajor, M. Rudelson, N. Tomczak-Jaegermann, R.~Vershynin,
{\it Random Euclidean embeddings in spaces of bounded volume ratio,}
C.R. Acad. Sci. Paris, Ser 1, Math., {\bf 339} (2004), 33--38.
}

\bibitem{MenP}
{
S. Mendelson, G. Paouris, {\it On singular
    values of matrices},  Journal of EMS,  {\bf 16} (2014),  823--834.
}



\bibitem{VN1}
 J. von Neumann. {\it Collected works. Vol. V: Design of computers, theory of automata and numerical analysis},
General editor: A. H. Taub. A Pergamon Press Book The Macmillan Co., New York 1963.



\bibitem{VN2}
{
 J. von Neumann, H.H. Goldstine. {\it Numerical inverting of matrices of high order}, Bull. Amer. Math. Soc.
{\bf 53} (1947), 1021--1099.
}

\bibitem{Ol}
{
R.I.~Oliveira,
{\it The lower tail of random quadratic forms, with applications to ordinary least squares
and restricted eigenvalue properties},
 PTRF, {\bf 166} (2016), 1175--1194.
}

\bibitem{RT}
{
E. Rebrova, K. Tikhomirov,
{\it Coverings of random ellipsoids, and invertibility of matrices with i.i.d.
heavy-tailed entries}, Israel J. Math., to appear. arXiv:1508.06690.
}



\bibitem{RV}
{
M. Rudelson, R. Vershynin,
{\it The Littlewood-Offord problem and invertibility of random matrices,}
Adv. Math.  {\bf 218} (2008), 600--633.
}

\bibitem{RV-comm}
{
M. Rudelson\ and\ R. Vershynin,
{\it Smallest singular value of a random rectangular matrix},
Comm. Pure Appl. Math. {\bf 62} (2009), no.~12, 1707--1739. MR2569075
}

\bibitem{RV-ICM}
{
M. Rudelson, R. Vershynin,
{\it  Non-asymptotic theory of random matrices: extreme singular values},
Proceedings ICM, Vol. III, 1576--1602, Hindustan Book Agency, New Delhi, 2010.
}

\bibitem{SST06}
{
A. Sankar, D. A. Spielman\ and\ S.-H. Teng,
{\it  Smoothed analysis of the condition numbers and growth factors of matrices},
SIAM J. Matrix Anal. Appl. {\bf 28} (2006), no.~2, 446--476 (electronic). MR2255338
}

\bibitem{Schech}
{
  G. Schechtman,
{\it   Special orthogonal splittings of $L_1^{2k}$},
  Isr. J. Math. {\bf 139} (2004), 337--347.
}

\bibitem{NLA2}
{
S. Smale, {\it On the efficiency of algorithms of analysis}, Bull. Amer. Math. Soc. (N.S.)
{\bf 13} (1985),
87--121.
}

\bibitem{ST02}
{
D. A. Spielman\ and\ S.-H. Teng, {\it Smoothed analysis of algorithms}, Proceedings ICM,
Vol. I, 597--606, Higher Ed. Press, Beijing, 2002.
}




\bibitem{SV13}
{
N. Srivastava\ and\ R. Vershynin, {\it Covariance estimation for distributions with $2+\varepsilon$ moments}, Ann. Probab.
{\bf 41} (2013), no.~5, 3081--3111. MR3127875
}

\bibitem{TV07b}
{
T. Tao\ and\ V. Vu, {\it The condition number of a randomly perturbed matrix},  STOC'07
-- Proceedings of the 39th Annual ACM Symposium on Theory of Computing, 248--255, ACM, New York, 2007. MR2402448
}



\bibitem{TV-ann}
{
T. Tao, V. Vu, {\it Inverse Littlewood-Offord theorems and the condition number of
random discrete matrices}, Annals of Math. {\bf 169} (2009), 595--632.
}

\bibitem{TV10c}
{
T. Tao\ and\ V. Vu, {\it Smooth analysis of the condition number and the least singular value},
Math. Comp. {\bf 79} (2010), no.~272, 2333--2352. MR2684367
}

\bibitem{TV}
{
T. Tao and V. Vu,
{\it Random matrices:  universality of ESDs and the circular law,}
Ann. Probab. {\bf 38(5)} (2010), 2023-2065. With an appendix by Manjunath Krishnapur.
}

\bibitem{T}
{
K. Tikhomirov,
{\it Sample covariance matrices of heavy-tailed distributions}, Int. Math. Res. Notes,
to appear, https://doi.org/10.1093/imrn/rnx067.
}


\bibitem{AGA3}
{
N. Tomczak-Jaegermann,
{\it Banach-Mazur distances and finite-dimensional operator ideals.}
Pitman Monographs and Surveys in Pure and Applied Mathematics, 38.
Longman Scientific \& Technical, Harlow; copublished in the United
States with John Wiley \& Sons, Inc., New York, 1989.
}


\bibitem{Vu-surv}
{
V.~Vu,
     {\it Random discrete matrices},
  {Horizons of combinatorics},
   {Bolyai Soc. Math. Stud.},
    {\bf 17}, {257--280},
 {Springer, Berlin}, {2008}.
}

\bibitem{Vu}
{
V.H.~Vu,
{\em  Combinatorial problems in random matrix theory},
Proceedings ICM, Vol. IV, 489--508, Kyung Moon Sa, Seoul, 2014.
}

\bibitem{Y1}
{
P.~Yaskov, {\it Lower bounds on the smallest eigenvalue of a sample covariance matrix},
Electron. Commun. Probab. 19 (2014), 1--10.
}

\bibitem{Y2}
{
P.~Yaskov,
{\it Sharp lower bounds on the least singular value of a random matrix without
the fourth moment condition},
Electron. Commun. Probab. {\bf 20} (2015), no. 44, 9 pp. MR3358966
}


\end{thebibliography}
\end{document}